\documentclass[12pt,letterpaper, reqno]{amsart}
\usepackage[margin=2.5cm]{geometry}
\usepackage{amscd}
\usepackage{amssymb}
\usepackage{amsthm}
\usepackage{graphicx}
\usepackage{color}
\usepackage[all]{xy}
\usepackage{mathrsfs}
\usepackage{marvosym}
\usepackage{stmaryrd}
\usepackage{srcltx}
\usepackage{mathtools}
\usepackage{scalerel,amssymb}
\usepackage{tikz-cd}
\usepackage{adjustbox}
\usepackage{placeins}
\usepackage{floatrow}
\usepackage{setspace}

\usepackage{todonotes}
\newfloatcommand{capbtabbox}{table}[][\FBwidth]
\definecolor{amaranth}{rgb}{0.9, 0.17, 0.31}
\usepackage[colorlinks=true,citecolor=amaranth,linkcolor=black]{hyperref}%
\usetikzlibrary{shapes,arrows,chains}
\usetikzlibrary{calc}

\let\oldtocsection=\tocsection
\let\oldtocsubsection=\tocsubsection
\let\oldtocsubsubsection=\tocsubsubsection
\renewcommand{\tocsection}[2]{\hspace{0em}\oldtocsection{#1}{#2}}
\renewcommand{\tocsubsection}[2]{\hspace{1em}\oldtocsubsection{#1}{#2}}
\renewcommand{\tocsubsubsection}[2]{\hspace{0.5em}\oldtocsubsubsection{#1}{#2}}
\newtheorem{theorem}{Theorem}[section]
\newtheorem{lemma}[theorem]{Lemma}
\newtheorem{proposition}[theorem]{Proposition}

\newtheorem{conjecture}[theorem]{Conjecture}

\newtheoremstyle{defstyle}
  {.6em} % Space above
  {.1em} % Space below
  {} % Body font
  {} % Indent amount
  {\bfseries} % Theorem head font
  {.} % Punctuation after theorem head
  {.5em} % Space after theorem head
  {} % Theorem head spec (can be left empty, meaning `normal')
\theoremstyle{defstyle} \newtheorem{definition}[theorem]{Definition}

\newtheorem{example}[theorem]{Example}

\theoremstyle{remark}
\newtheorem{remark}[theorem]{Remark}

\numberwithin{equation}{section}
\numberwithin{figure}{section}

\newcommand{\cF} {\mathcal{F}}

\newcommand{\cC} {\mathcal{C}}
\newcommand{\Spec} {\mathrm{Spec}}
\newcommand{\Proj} {\mathrm{Proj}}
\newcommand{\gr} {\mathrm{gr}}

\newcommand{\kk} {\Bbbk} 
\newcommand{\A} {\mathbb{A}}
\newcommand{\CC} {\mathbb{C}}

\newcommand{\G} {\mathbb{G}}
\newcommand{\bS} {\mathbb{S}}
\newcommand{\Hom} {\mathrm{Hom}}

\newcommand{\cO} {\mathcal{O}}

\newcommand{\PP} {\mathbb{P}}
\newcommand{\QQ} {\mathbb{Q}}
\newcommand{\RR} {\mathbb{R}}

\newcommand{\ZZ} {\mathbb{Z}}

\newcommand{\sP} {\mathscr{P}}
\newcommand{\cY} {\mathcal{Y}}
\newcommand{\cD} {\mathcal{D}}

\newcommand{\cX} {\mathcal{X}}
\newcommand{\cP} {\mathcal{P}}

\begin{document}

%\title[Log Calabi--Yau compactifications of $SL(2,\CC)$ character varieties]{Log Calabi--Yau compactifications of $SL(2,\CC)$ character varieties}

\title[Log Calabi--Yau compactifications of $SL(2,\CC)$ character varieties]{\fontsize{10}{1}\selectfont Log Calabi--Yau compactifications of $SL(2,\CC)$ character varieties}

\author[H.\,Arg\"uz]{\fontsize{10}{1}\selectfont H\"ulya Arg\"uz}
\address{The Mathematical Institute, University of Oxford, Oxford, OX2 6GG, UK}
\email{hulya.arguz@maths.ox.ac.uk}

\author[P.\,Bousseau]{Pierrick Bousseau}
\address{The Mathematical Institute, University of Oxford, Oxford, OX2 6GG, UK}
\email{pierrick.bousseau@maths.ox.ac.uk}

\begin{abstract}
We prove that the $SL(2,\mathbb{C})$ character varieties of compact oriented surfaces and the generic relative $SL(2,\mathbb{C})$ character varieties of punctured surfaces admit divisorial log terminal (dlt) log Calabi--Yau compactifications. To do this, we establish a general result giving sufficient conditions for a compactification of an affine variety arising from a filtration of its algebra of regular functions to be log Calabi--Yau. We then apply this result to show that the compactifications constructed by Kutteri--Tehrani--Frohman in the compact case and by Tehrani--Frohman in the punctured case are log canonical and log Calabi--Yau.

\end{abstract}
\newtheoremstyle{cited}%
  {3pt}% (space above)
  {3pt}% (space below)
  {\itshape}% (body font)
  {}% (indent amount)
  {\bfseries}% {theorem head font}
  {.}% {punctuation after theorem head}
  {.3em}% {space after theorem head}
  {\thmname{#1} \thmnumber{#2}\thmnote{\normalfont#3}}% {theorem head spec}

\theoremstyle{cited}
\newtheorem{citedthm}{Theorem}
\renewcommand*{\thecitedthm}{\Alph{citedthm}}

\maketitle

\setcounter{tocdepth}{1}
\tableofcontents
\setcounter{section}{0}

\section{Introduction}

Character varieties are affine complex algebraic varieties parameterizing conjugacy classes of representations of the fundamental group of a topological surface into a complex reductive algebraic group \cite{LM, sikora}. They lie at the crossroads of low-dimensional topology, algebraic geometry, and mathematical physics. 
In particular, character varieties play an important role in non-abelian Hodge theory \cite{simpson1994moduliI, simpson1994moduliII}, the Betti formulation of the geometric Langlands program \cite{ben2016betti}, and the study of supersymmetric gauge theories \cite{GMN_spectral}.

A folklore conjecture predicts that character varieties admit log Calabi--Yau compactifications with relatively mild singularities \cite{GHKK, KNPS, KS_WCS, simpson2016dual}; see Conjecture \ref{conj_main} for a precise formulation. In this paper, we prove this conjecture for $SL(2,\CC)$ character varieties of compact oriented surfaces and for relative $SL(2,\CC)$ character varieties of punctured oriented surfaces. Our approach is based on a general result giving sufficient conditions for a compactification of an affine variety, arising from a filtration of its algebra of regular functions, to be log Calabi--Yau. We apply this result to the compactifications recently constructed by Kutteri--Tehrani--Frohman \cite{ayilliath2025geometric} in the compact case and by Tehrani--Frohman \cite{farajzadeh2023compactifications} in the punctured case, showing that they are log Calabi--Yau and have log canonical singularities.

After briefly reviewing character varieties of surfaces in \S\ref{sec_intro_character} and log Calabi--Yau compactifications in \S\ref{sec_intro_log_symplectic}, we state our main results in \S\ref{sec_intro_main_results}.

\subsection{Character varieties of surfaces}
\label{sec_intro_character}
For each integer \( g \ge 0 \), let \( \bS_g \) be the compact, connected, oriented topological surface of genus \( g \).
For every complex connected reductive algebraic group $G$, the space
\[
\Hom(\pi_1(\bS_g), G)
\]
of representations of the fundamental group \( \pi_1(\bS_g) \) in \( G \) naturally carries the structure of an affine variety over $\CC$. The group \( G \) acts on this space by conjugation and
the \emph{\( G \)-character variety} of \( \bS_g \) is defined as the affine geometric invariant theory (GIT) quotient, that is, as the Spec of the algebra of $G$-invariant regular functions \cite{LM, sikora},
\[
\mathrm{Ch}_g^G \coloneqq \Hom(\pi_1(\bS_g), G)\ /\!/\, G\,.
\]
The affine variety $\mathrm{Ch}_g^G$ has connected components indexed by the finite group $\pi_1(G_1)$, where $G_1$ is the derived subgroup of $G$ \cite{shu2024singularities}. These connected components are typically singular. They are known to be normal and $\QQ$-factorial,
of dimension $(2g-2) \dim G +2 \dim Z(G)$, where $Z(G)$ is the center of $G$, and to have Gorenstein canonical singularities \cite{shu2024singularities}.

Similarly, for all integers $g\geq 0$ and $n>0$, let $\bS_{g,n}$
be the complement of $n$ points in $\bS_g$. Then, the $G$-character variety $\mathrm{Ch}_{g,n}^G \coloneqq \Hom(\pi_1(\bS_{g,n}), G)\ /\!/\, G$ admits a natural map 
\[ \pi: \mathrm{Ch}_{g,n}^G \longrightarrow (G\ /\!/\, G)^n\]
given by the conjugacy classes of elements in $G$
assigned to small oriented loops around the punctures.
The fibers $\mathrm{Ch}_{g,n}^{G,t}:=\pi^{-1}(t)$ over 
$t \in  (G\ /\!/\, G)^n$ are affine varieties referred to as the \emph{relative $G$-character varieties} of $\bS_{g,n}$.

\subsection{Log Calabi--Yau compactifications}
\label{sec_intro_log_symplectic}

Since character varieties are affine, understanding their geometry at infinity naturally leads to the study of their projective compactifications. A folklore conjecture, motivated in part by mirror symmetry, predicts that character varieties admit log Calabi--Yau compactifications \cite{GHKK, KNPS, KS_WCS, simpson2016dual}. By a \emph{log Calabi--Yau compactification} of a normal variety $X$, we mean a pair $(Y,D)$, where $Y$ is a normal projective variety and $D$ is a reduced effective Weil divisor on $Y$, such that $X=Y \setminus D$ and $(Y,D)$ has trivial log canonical divisor class, that is, $K_Y +D \sim 0$.
The existence of a log Calabi--Yau compactification imposes strong restrictions on an affine variety. A necessary, though generally not sufficient, condition is that $K_X\sim 0$. Already in dimension one, the existence of such a compactification is highly restrictive: the only normal affine curve admitting a log Calabi--Yau compactification is $\CC^\star=\PP^1\setminus\{0,\infty\}$.

From the perspective of the Minimal Model Program, it is natural to further require the compactifying pair $(Y,D)$ to have mild singularities, such as log canonical or divisorial log terminal (dlt) singularities. These classes of singularities arise naturally in birational geometry: log canonical and dlt singularities are precisely the types of singularities that occur on canonical and minimal models of simple normal crossing pairs, respectively. We refer to \cite{kollar2013singularities, KM} for the general definitions of these singularities and their role in birational geometry. In the case of log Calabi--Yau pairs, these notions admit a particularly useful description in terms of volume forms.

Indeed, let $(Y,D)$ be a log Calabi--Yau compactification of a normal variety $X$. The relation $K_Y+D\sim 0$ implies the existence of a holomorphic volume form $\Omega$ on the smooth locus of $X$ with a first order pole along every irreducible component of $D$. Let $\pi:\widetilde{Y}\to Y$ be a log resolution, namely a projective birational morphism from a smooth projective variety $\widetilde{Y}$, such that the union $\widetilde{D}$ of the strict transforms of the irreducible components of $D$ together with the exceptional divisors of $\pi$ is a simple normal crossing divisor on $\widetilde{Y}$. Then $(Y,D)$ is log canonical if and only if the pullback $\pi^\star\Omega$ has at most first order poles along every irreducible component of $\widetilde{D}$. Moreover, $(Y,D)$ is dlt if, in addition, the image $\pi(\widetilde{D}_i)$ of every irreducible component $\widetilde{D}_i$ along which $\pi^\star\Omega$ has a first order pole is generically contained in the simple normal crossing locus of $(Y,D)$.

Log canonical singularities form a broader and, in many respects, more natural class for the general study of log Calabi--Yau varieties. Nevertheless, as reviewed in Lemma \ref{lem_dlt_modifications}, for $\QQ$-factorial varieties with canonical singularities, the existence of a log canonical log Calabi--Yau compactification in fact implies the existence of a dlt log Calabi--Yau compactification. Since $\mathrm{Ch}_g^G$ is known to be $\QQ$-factorial with canonical singularities \cite{shu2024singularities}, it is natural to refine the folklore conjecture for compact surfaces by requiring the log Calabi--Yau compactification to be dlt, as in \cite[Conjecture 1.3.1]{mauri2022geometric}. By contrast, it is not currently known whether the relative character varieties $\mathrm{Ch}_{g,n}^{G,t}$ are always $\QQ$-factorial with canonical singularities. We therefore refine the conjecture for punctured surfaces by requiring only that the log Calabi--Yau compactification be log canonical:

\begin{conjecture}\label{conj_main}
Let $G$ be a complex reductive group. Then the following hold:
\begin{itemize}
\item[(i)] \emph{(Compact case)}
For every integer $g\geq 0$, each connected component of the $G$-character variety $\mathrm{Ch}_g^G$ admits a dlt log Calabi--Yau compactification.

\item[(ii)] \emph{(Punctured case)}
For every pair of integers $g\geq 0$ and $n>0$, and every $t\in (G//G)^n$, each connected component of the relative $G$-character variety $\mathrm{Ch}_{g,n}^{G,t}$ admits a log canonical log Calabi--Yau compactification.
\end{itemize}
\end{conjecture}

For $G=GL(1,\CC)=\CC^\star$, the 
conjecture is straightforward: both $\mathrm{Ch}_g^{\CC^\star}$ and $\mathrm{Ch}_{g,n}^{\CC^\star,t}$ are isomorphic to $(\CC^\star)^{2g}$, and hence admit smooth projective toric compactifications, which are dlt log Calabi--Yau. Thus, Conjecture \ref{conj_main} may be viewed as seeking non-toric analogues of this familiar picture for non-abelian groups.

The first nontrivial examples in the non-abelian setting arise for $G=SL(2,\CC)$ and the particular pairs $(g,n)=(0,4)$ and $(g,n)=(1,1)$, where the trace induces an isomorphism
$(SL(2,\CC) \ /\!/\, SL(2,\CC))^n \simeq \CC^n$. In these cases \cite{fricke_klein, goldman2010affine}, the relative character varieties $\mathrm{Ch}_{g,n}^{SL(2,\CC),t}$ are affine cubic surfaces $X$ in $\CC^3$ of the form
\[xyz+x^2+y^2+z^2=f_t(x,y,z)\,.\]
For  $(g,n)=(1,1)$ and $t\in \CC$, one has $f_t(x,y,z)=t+2$, whereas for $(g,n)=(0,4)$ and $t=(t_1,t_2,t_3,t_4)\in\CC^4$, one has
\[ f_t(x,y,z)= (t_1t_2 +t_3t_4)x+(t_2t_3 +t_4t_1)y +(t_3t_1+t_2t_4)z +(4-t_1^2-t_2^2-t_3^2-t_4^2-t_1t_2t_3t_4) \,.\]
In these two cases, the closure $Y$ of $X$ in $\PP^3$ is a projective cubic surface, and the boundary divisor $D=Y\setminus X$ is a triangle of lines. Consequently, $K_Y+D \sim 0$, so $(Y,D)$ is a log Calabi--Yau compactification of $X$. Moreover, by \cite[Theorem 1]{goldman2010affine}, $Y$ is smooth in a neighborhood of $D$, the divisor $D$ is simple normal crossing, and $X$ has at most canonical singularities. Thus, $(Y,D)$ is log canonical.

\subsection{Main results}
\label{sec_intro_main_results}

Our first main result, stated below as Theorem \ref{thm_intro_A} and proved in the main body as Theorem \ref{thm_main_A}, establishes Conjecture \ref{conj_main} for $G=SL(2,\CC)$ for arbitrary $g$ and $n$. In the punctured case, we obtain a stronger result when the prescribed boundary traces $t\in\CC^n$ satisfy a natural genericity condition. Following \cite[Condition 4.3]{simpson2016dual}, based on \cite{Kostov}, we say that $t=(t_1,\ldots,t_n)\in\CC^n$ is \emph{Kostov generic} if, writing
$t_i=\lambda_i+\lambda_i^{-1}$ for all $1\le i\le n$,
there is no choice of signs $\epsilon_1,\ldots,\epsilon_n\in \{\pm1\}$ such that
$
\prod_{i=1}^n \lambda_i^{\epsilon_i}=1$. 
By \cite[Proposition 4.5]{simpson2016dual}, Kostov genericity implies that the relative character variety $\mathrm{Ch}_{g,n}^{SL(2,\CC),t}$ is smooth.

\begin{citedthm}
\label{thm_intro_A}
Conjecture \ref{conj_main} holds for $G=SL(2,\CC)$. More precisely:
\begin{itemize}
\item[(i)] \emph{(Compact case)}
For every integer $g\geq 0$, the $SL(2,\CC)$-character variety $\mathrm{Ch}_g^{SL(2,\CC)}$ admits a dlt log Calabi--Yau compactification.

\item[(ii)] \emph{(Punctured case)}
For all integers $g\geq 0$ and $n>0$, and every $t\in\CC^n$, the relative $SL(2,\CC)$-character variety $\mathrm{Ch}_{g,n}^{SL(2,\CC),t}$ admits a log canonical log Calabi--Yau compactification. If, moreover, $t$ is Kostov generic, then $\mathrm{Ch}_{g,n}^{SL(2,\CC),t}$ admits a dlt log Calabi--Yau compactification.
\end{itemize}
\end{citedthm}

The proof of Theorem \ref{thm_intro_A} relies on a detailed analysis of compactifications of $SL(2,\CC)$ character varieties recently constructed by Kutteri--Tehrani--Frohman \cite{ayilliath2025geometric} in the compact case and by Tehrani--Frohman \cite{farajzadeh2023compactifications} in the punctured case. More precisely, for every integer $g\geq 2$ and every choice $\cP=(P,\Gamma)$ of a pair of pants decomposition $P$ of $\bS_g$ together with an embedded dual graph $\Gamma$, Kutteri--Tehrani--Frohman construct a compactification $(Y_\cP,D_\cP)$ of the character variety $\mathrm{Ch}_g^{SL(2,\CC)}$. Similarly, for every $g\geq 0$ and $n>0$ with $2g-2+n>0$, and every ideal triangulation $\Delta$ of $\bS_{g,n}$, Tehrani--Frohman construct a compactification $(Y_\Delta,D_\Delta)$ of the relative character variety $\mathrm{Ch}_{g,n}^{SL(2,\CC),t}$.

Our next main result establishes the key geometric properties of these compactifications, showing that they are log canonical and log Calabi--Yau. This result, which provides the main geometric input for the proof of Theorem \ref{thm_intro_A}, is stated here as Theorem \ref{thm_intro_B} and proved in the main body as Theorem \ref{thm_main_B}.

\begin{citedthm}
\label{thm_intro_B} The following hold:
\begin{itemize}
\item[(i)] \emph{(Compact case)}
Let $g\geq 2$, and let $\cP=(P,\Gamma)$ consist of a pair of pants decomposition $P$ of $\bS_g$ together with an embedded dual graph $\Gamma$. Then the compactification $(Y_\cP,D_\cP)$ of $\mathrm{Ch}_g^{SL(2,\CC)}$ is log canonical and log Calabi--Yau.

\item[(ii)] \emph{(Punctured case)}
Let $g\geq 0$ and $n>0$ satisfy $2g-2+n>0$, let $t\in\CC^n$, and let $\Delta$ be an ideal triangulation of $\bS_{g,n}$. Then the compactification $(Y_\Delta,D_\Delta)$ of $\mathrm{Ch}_{g,n}^{SL(2,\CC),t}$ is log canonical and log Calabi--Yau.
\end{itemize}
\end{citedthm}

As explained in \S\ref{sec_lcy_character}, Theorem \ref{thm_intro_A} follows from Theorem \ref{thm_intro_B} together with general results from the Minimal Model Program that allow one to pass from log canonical to dlt log Calabi--Yau compactifications -- see Lemma \ref{lem_dlt_modifications}. The proof of Theorem \ref{thm_intro_B} is an application of Theorem \ref{thm_intro_C} below, referred to as Theorem \ref{thm_main_C} in the main body of the paper, that provides a general criterion of independent interest for showing that certain compactifications are log canonical and log Calabi--Yau.

The compactifications $(Y_\cP,D_\cP)$ and $(Y_\Delta,D_\Delta)$ are constructed in \cite{ayilliath2025geometric, farajzadeh2023compactifications} as Proj of Rees algebras associated with filtrations of the corresponding algebras of regular functions. As reviewed in \S\ref{sec_compactifications_SL2}, their associated graded algebras admit a simple combinatorial description as gluings of monoid algebras of integral points in the cones of a cone complex whose underlying topological space is homeomorphic to a cone over a sphere. Theorem \ref{thm_intro_C} shows that these properties are sufficient to guarantee that the resulting compactifications are log canonical and log Calabi--Yau.

To state the theorem, we fix an algebraically closed field $\kk$ of characteristic zero. Projective filtrations and their associated projective compactifications, as well as degree-like functions and subdegrees, are reviewed in \S\ref{sec_filtrations}, while the algebra $\kk[\Omega]$ associated with a cone complex $\Omega$ is recalled in \S\ref{sec_algebras_cone}.

\begin{citedthm}
\label{thm_intro_C}
Let $\cF=(F_d)_{d\geq 0}$ be a projective filtration on a $\kk$-algebra $R$, and let $(Y,D)$ be the associated projective compactification of $X=\Spec(R)$. Assume that:
\begin{itemize}
\item[(i)] $R$ is an integrally closed domain, equivalently, $X$ is a normal variety.
\item[(ii)] The degree-like function $\delta$ defining $\cF$ is a subdegree.
\item[(iii)] There exist a cone complex $\Omega$ and a non-negative integral piecewise linear function $ h:|\Omega|\longrightarrow \RR_{\geq 0}$ such that the intersection of any two cones of $\Omega$ is a cone of $\Omega$, the level set $h^{-1}(1)$ is homeomorphic to a sphere, and the algebra $\kk[\Omega]$, endowed with the grading induced by $h$, is isomorphic to $\gr_{\cF}(R)$ as a $\ZZ_{\geq 0}$-graded algebra.
\end{itemize}
Then:
\begin{itemize}
\item[(i)] $Y$ is a normal projective variety and $D$ is a reduced effective Weil divisor on $Y$,
\item[(ii)] $Y$ is Cohen--Macaulay and $(Y,D)$ is log canonical, and
\item[(iii)] $(Y,D)$ is log Calabi--Yau, that is, $K_Y+D\sim 0$.
\end{itemize}
\end{citedthm}

The proof of Theorem \ref{thm_intro_C} is based on a degeneration of $(Y,D)$ whose central fiber is a union of toric varieties. This degeneration is constructed as the Proj of a doubly graded algebra and admits a geometric interpretation as a base change along the map $t \mapsto xy$, following a construction previously used by the first author in \cite[Theorem 3.2]{Arguzcorals}. We show that the central fiber is semi log canonical and log Calabi--Yau, and then use a deformation argument to deduce that $(Y,D)$ is log canonical and log Calabi--Yau. Degenerations to unions of toric varieties have played an important role in mirror symmetry \cite{AAB2024ksba, GHK1, GHKK, GHS, GS2019intrinsic, HKY20, keel2024log}, where they arise in constructions of mirror varieties as deformations of such unions. Our proof draws on techniques and results developed in this setting, particularly those of \cite{GHKK, keel2024log}.

\subsection{Outline of the paper} 

In \S\ref{sec_filtrations}, following \cite{mondal2014projective}, we review the general framework for constructing projective compactifications of affine varieties from filtrations of their algebras of regular functions. In \S\ref{sec_compactifications_SL2}, we recall how this framework is used in \cite{ayilliath2025geometric, farajzadeh2023compactifications} to construct compactifications of $SL(2,\CC)$ character varieties.

In \S\ref{sec_algebras_cone}, we introduce the algebras associated with cone complexes. This leads, in \S\ref{sec_lcy_statement}, to the statement of Theorem~\ref{thm_intro_C}, our general criterion for obtaining log Calabi--Yau compactifications from filtrations. We apply this criterion to the compactifications of character varieties in \S\ref{sec_lcy_character}, proving Theorems~\ref{thm_intro_A} and \ref{thm_intro_B}.

The remainder of the paper is devoted to the proof of Theorem~\ref{thm_intro_C}. In \S\ref{section_degeneration}, we construct canonical degenerations associated with compactifications arising from filtrations, and in \S\ref{section_vertices}, we analyze their central fibers under the hypotheses of Theorem~\ref{thm_intro_C}. In \S\ref{section_extending}, we show that the relevant properties of the central fiber extend to the general fiber. The proof of Theorem~\ref{thm_intro_C} is completed in \S\ref{section_end_proof}.

\subsection{Related works and future directions}

\subsubsection{Compactifications of character varieties}

In the punctured case, for $n>0$, log Calabi--Yau compactifications of $\mathrm{Ch}_{g,n}^{SL(2,\CC),t}$ were constructed in \cite{MR4157427}. However, the singularities of these compactifications do not appear to have been studied; in particular, it is not known whether they are log canonical.
In the compact case, for $g=1$, dlt log Calabi--Yau compactifications of $\mathrm{Ch}_1^{GL(k,\CC)}$ and $\mathrm{Ch}_1^{SL(k,\CC)}$ were constructed for all $k\geq 1$ in \cite[Theorem B]{mauri2022geometric}. Recall that $\mathrm{Ch}_1^{GL(k,\CC)}$ is simply $\mathrm{Sym}^k((\CC^\star)^2)$. In this case, a log canonical log Calabi--Yau compactification can be obtained directly: for any smooth projective toric compactification $Z$ of $(\CC^\star)^2$, the symmetric product $\mathrm{Sym}^k(Z)$ provides such a compactification of $\mathrm{Sym}^k((\CC^\star)^2)$.
Theorem \ref{thm_intro_A} gives the first examples of log Calabi--Yau compactifications of $\mathrm{Ch}_g^{SL(2,\CC)}$ for $g\geq 2$, as well as the first log canonical log Calabi--Yau compactifications of $\mathrm{Ch}_{g,n}^{SL(2,\CC),t}$ beyond the elementary cases $(g,n)=(1,1)$ and $(g,n)=(0,4)$ discussed at the end of \S\ref{sec_intro_log_symplectic}.

\subsubsection{Dual complexes and the geometric $P=W$ conjecture}
Conjecture \ref{conj_main} is part of a broader program aimed at understanding the geometry at infinity of character varieties \cite{KNPS, mauri2022geometric}. One expects, in particular, the dual complex of any dlt log Calabi--Yau compactification of a character variety to be homeomorphic to a sphere \cite[Conjecture 1.3.1]{mauri2022geometric}. The geometric $P=W$ conjecture goes further, predicting a natural identification of this dual complex with the sphere at infinity in the base of the Hitchin fibration, compatible with the non-abelian Hodge correspondence; see \cite{KNPS} and \cite[Conjecture 4.2.7]{mauri2022geometric} for further details. Theorem~\ref{thm_intro_A} provides a first step toward investigating these predictions for $SL(2,\CC)$ character varieties.

\subsubsection{The Goldman volume form} 
Given a log Calabi--Yau compactification $(Y,D)$ of a normal affine variety $X$, there is, up to a nonzero scalar, a unique algebraic volume form on the smooth locus of $X$ with first order poles along every irreducible component of $D$. On the other hand, the smooth loci of character varieties of compact surfaces and of relative character varieties of punctured surfaces carry a natural algebraic volume form, the \emph{Goldman volume form}, obtained as the top exterior power of a Goldman symplectic form \cite{AMM, atiyahbott, boalch, FR, goldman, GHJW, sikora}. While the Goldman symplectic form depends on the choice of a nondegenerate symmetric $\operatorname{Ad}$-invariant bilinear form on the Lie algebra of $G$, its top exterior power is unique up to an overall scalar. It is natural to expect that the volume forms arising from the log Calabi--Yau compactifications of Theorems~\ref{thm_intro_A}--\ref{thm_intro_B} coincide, up to scale, with the Goldman volume form. A natural refinement of this expectation is that the Goldman symplectic form itself has first order poles along every irreducible component of $D$. This would in particular show that character varieties admit log symplectic compactifications, as predicted in \cite[Conjecture 5.4]{mauriintro}. We leave these questions for future investigation.

\subsection{Acknowledgments} 
This paper was completed during the authors’ stay at the Max Planck Institute for Mathematics in Bonn in July 2026, and they thank the Institute for its hospitality and excellent working environment. They also greatly appreciate Dominic Joyce’s valuable comments and suggestions. The second author acknowledges the support of a Sloan Research Fellowship from the Alfred P. Sloan Foundation. For the purpose of open access, the authors have applied a CC BY public copyright licence to any author accepted manuscript arising from this submission.

\section{Compactifications from filtrations}
\label{sec_compactification}
Throughout this paper, we work over an algebraically closed field \(\kk\) of characteristic \(0\), and $\kk$-algebras are always commutative with unit. 
In \S\ref{sec_filtrations}, we review the construction of projective compactifications of affine varieties arising from filtrations on their algebras of regular functions, following primarily \cite{mondal2014projective}; see also \cite[\S 2.1]{farajzadeh2023compactifications} for a concise overview, as well as \cite{demazure1988anneaux,watanabe} for related background. We then describe in \S\ref{sec_compactifications_SL2} the projective compactifications of $SL(2,\CC)$ character varieties associated with particular filtrations, following the constructions of \cite{ayilliath2025geometric,farajzadeh2023compactifications}.

\subsection{Filtrations and projective compactifications}
\label{sec_filtrations}

\subsubsection{Filtrations, degree-like functions and Rees algebras}

\begin{definition} Let $R$ be a $\kk$-algebra. A \emph{filtration} $\mathcal{F}$ on $R$ is a family
$(F_d)_{d \in \ZZ}$
of $\kk$-vector subspaces of $R$ such that:
\begin{itemize}
    \item[(i)] \(F_d \subseteq F_{d+1}\) for all \(d \in \mathbb{Z}\).
    \item[(ii)] \(1 \in F_0 \setminus F_{-1}\).
    \item[(iii)] $R = \bigcup_{d \in \mathbb{Z}} F_d$, and
    \item[(iv)] $F_d F_{d'} \subseteq F_{d+d'}$ for all \(d,d' \in \mathbb{Z}\).
\end{itemize}
\end{definition}

Equivalently, a filtration on a \(\kk\)-algebra \(R\) is encoded by a \emph{degree-like function} on \(R\), which we now recall.

\begin{definition}
Let $R$ be a $\kk$-algebra. A \emph{degree-like function} on $R$ is a map 
\[\delta: R \longrightarrow \ZZ \cup \{-\infty\} \,\]
satisfying the following conditions:
\begin{itemize}
    \item[(i)] $\delta(0)=-\infty$ and $\delta(\kk^\star)=0$.
    \item[(ii)] $\delta(f+g) \leq \max(\delta(f),\delta(g))$, and $\delta(f)=\delta(g)$ if $\delta(f+g) < \max(\delta(f),\delta(g))$.
    \item[(iii)] $\delta(fg) \leq \delta(f) + \delta(g)$ for all $f,g \in R$.
\end{itemize}
\end{definition}

A degree-like function $\delta$ defines a filtration $\cF=(F_d)_{d\in \ZZ}$ on $R$, where $F_d :=\{ f \in R\,|\, \delta(f) \leq d \}$. 
Conversely, any filtration $\cF=(F_d)_{d \in \ZZ}$ on $R$ determines a degree-like function $\delta$ via $\delta(f):=\inf \{d \in \ZZ \,|\, f \in F_d \}$.

\begin{example}
    \label{Ex:degree like function}
Let \(R=\kk[x_1,\dots,x_n]\). The usual polynomial degree defines a degree-like function
\[ \delta(f)=\deg(f)
\]
for every nonzero polynomial \(f \in R\). In the induced filtration,
$F_d^{\delta}$ consists of polynomials of total degree at most \(d\). 
%The associated projective completion is the standard compactification
%\[
%\mathbb{A}^n \hookrightarrow \mathbb{P}^n.
%\]
\end{example}

\begin{definition}
    A filtration $\cF=(F_d)_{d\in \ZZ}$ on a $\kk$-algebra $R$ is \emph{non-negative} if $F_d =0$ for all $d<0$ and $F_0=\kk$. A non-negative filtration will usually be denoted by $\cF=(F_d)_{d\geq 0}$.
\end{definition}

\begin{definition}
    Let $\cF=(F_d)_{d \geq 0}$ be a non-negative filtration on a $\kk$-algebra $R$. The  \emph{Rees algebra} of $\cF$
    is the $\ZZ_{\geq 0}$-graded $\kk[t]$-algebra
    \begin{equation} \label{eq_rees}
    \mathrm{Rees}_\cF(R) := \bigoplus_{d \geq 0} F_d t^d \subset R[t]\,.\end{equation}
 The \emph{associated graded algebra} of $\cF$ is the $\ZZ_{\geq 0}$-graded $\kk$-algebra
\[ \mathrm{gr}_\cF(R) := \mathrm{Rees}_\cF(R)/(t) = \bigoplus_{d\geq 0} F_d /F_{d-1}  \,.\]
\end{definition}

\begin{proposition}
 Let $\cF=(F_d)_{d \geq 0}$ be a non-negative filtration on a $\kk$-algebra $R$. Then $t$ is a non-zero divisor of  $\mathrm{Rees}_\cF(R)$. Moreover, 
 \[  \mathrm{Rees}_\cF(R)  \otimes_{\kk[t]} \kk[t^\pm] \simeq R[t^\pm]\,\,\,\,\,\text{and}\,\,\,\,\,  \mathrm{Rees}_\cF(R)/(t-1) \simeq R \,. \]
\end{proposition}

\begin{proof}
This is standard -- see for instance \cite[I.9.5 (i)]{MR1321145}.
\end{proof}

%Note that $T$ is a non-zero divisor in $A^\delta$, and the original algebra $A$ can be recovered as the degree-zero part of the localization $A^\delta[\frac{1}{T}]$:
%\[ A = \left(A^\delta \bigg[\frac{1}{T}\bigg] \right)^0 \,.\]

\begin{definition}
A non-negative filtration $\cF$ on a $\kk$-algebra $R$ is \emph{finitely generated} if $\mathrm{Rees}_\cF(R)$ is finitely generated as a $\kk$-algebra.
\end{definition}

\begin{proposition} \label{prop_finitely_generated}
Let $\cF=(F_d)_{d \geq 0}$ be a non-negative filtration on a $\kk$-algebra $R$. Then, $\cF$ is finitely generated if and only if $\gr_\cF(R)$ is finitely generated as a $\kk$-algebra.
\end{proposition}

\begin{proof}
The ``only if" direction is clear since $\gr_\cF(R)$ is a quotient of $\mathrm{Rees}_\cF(R)$.
For the ``if" direction, let  $\overline{a}_1,\dots,\overline{a}_n$ be homogeneous elements of degree $d_1, \dots, d_n$ generating $\gr_\cF(R)$. Then, by induction on the filtration degree,  $\mathrm{Rees}_\cF(R)$ is generated by $a_1 t^{d_1}, \dots, a_n t^{d_n}$, where $a_i \in F_{d_i}$ is a lift of $\overline{a}_i$ for all $1 \leq i \leq n$ -- see for instance
\cite[I.9.5 (ii)]{MR1321145}.
\end{proof}

\begin{definition}
A \emph{projective} filtration on a $\kk$-algebra $R$ is a finitely generated non-negative filtration $\cF=(F_d)_{d\geq 0}$ on $R$ such that $F_0=\kk$.
\end{definition}

\begin{example}
Let \(R=\kk[x_1,\dots,x_n]\) be equipped with the standard degree filtration \(\mathcal{F}\) from Example~\ref{Ex:degree like function}. The associated Rees algebra is
\[
\operatorname{Rees}_{\mathcal{F}}(R)
=\bigoplus_{d\geq 0}F_dt^d
\simeq
\bigoplus_{d\geq 0}\kk[t,y_1,\dots,y_n]_d,
\]
where \(y_i=x_it\) for \(1\le i\le n\), and \(\kk[t,y_1,\dots,y_n]_d\) denotes the \(\kk\)-vector space of homogeneous polynomials of degree \(d\) in the variables \(t,y_1,\dots,y_n\).
\end{example}

\subsubsection{Projective compactifications}

\begin{definition} \label{def_projective_compactification}
    Let $\cF$ be a projective filtration on a $\kk$-algebra $R$. 
    The \emph{projective compactification} of $X:=\Spec(R)$ defined by $\cF$ is the pair $(Y,D)$, where
\[Y:= \Proj\left( \mathrm{Rees}_\cF(R)\right) \,\,\,\,\,\text{and}\,\,\,\,\, D:=\Proj \left( \gr_\cF(R) \right)\,.\]
\end{definition}

\begin{proposition}
    Let  $\cF$ be a projective filtration on a $\kk$-algebra $R$, and $(Y,D)$ the corresponding projective compactification of $X=\Spec(R)$. Then, the following hold:
    \begin{itemize}
        \item[(i)] $Y$ and $D$ are projective schemes over $\kk$.
        \item[(ii)] The quotient map $\mathrm{Rees}_\cF(R) \rightarrow \gr_\cF(R)$ induces a closed embedding $D \hookrightarrow Y$.
        \item[(iii)] The open complement $Y \setminus D$ is isomorphic to $X$.
    \end{itemize}
\end{proposition}

\begin{proof}
See for instance \cite[Proposition 2.8]{mondal2014projective}.
\end{proof}

As explained below, the geometry of $(Y,D)$ can be described much more precisely when the degree-like function defining the filtration is a ``subdegree'' defined as follows (see  \cite[Definition 3.3]{mondal2014projective}).

\begin{definition}
A degree-like function $\delta$ on a $\kk$-algebra $R$ is a \emph{semidegree} if $\delta(fg) =\delta(f)+\delta(g)$ for all $f,g \in R \setminus \{0\}$. More generally, $\delta$ is a \emph{subdegree} if there  exist finitely many semidegrees $\delta_1,\dots,\delta_{\ell}$ such that $\delta=\max_{1 \leq i \leq \ell} \delta_i$.
\end{definition}

\begin{proposition}
  Let $\cF$ be a projective filtration on a $\kk$-algebra $R$, and $\delta$ the corresponding degree-like function. Assume that $R$ is a domain. Then, the following are equivalent:
  \begin{itemize}
      \item[(i)] $\delta$ is a subdegree.
      \item[(ii)] $\delta(f^k)=k \delta(f)$ for all $f \in R$ and $k \in \ZZ_{\geq 0}$.
      \item[(iii)] The projective scheme $D=\mathrm{Proj} \left( \gr_\cF(R) \right)$ is reduced.
  \end{itemize}
\end{proposition}

\begin{proof}
The equivalence between (i) and (iii) is \cite[Theorem 4.1]{mondal2014projective}.
The equivalence between (i) and (ii) is
\cite[Corollary 4.2]{mondal2014projective}.
\end{proof}

The following proposition summarizes the properties of the projective compactifications established in 
\cite{mondal2014projective} and that will be used in \S\ref{sec_compactifications_SL2}.

\begin{proposition} \label{prop_proj_comp}
      Let $\cF=(F_d)_{d \geq 0}$ be a projective filtration on a $\kk$-algebra $R$ and $(Y,D)$ the associated projective compactification of $X=\Spec(R)$. Assume that $R$ is an integrally closed domain, that is, that $X$ is a normal variety, and that 
       the degree-like function  $\delta$ defining $\cF$ is a subdegree. Then, the following hold:
      \begin{itemize}
          \item[(i)] $\mathrm{Rees}_\cF(R)$ is an integrally closed domain and $Y$ is a normal projective variety.
          \item[(ii)] $D$ is a reduced effective Weil divisor on $Y$.
          \item[(iii)] There exists a unique minimal presentation $\delta=\max_{1 \leq i \leq \ell} \delta_i$, where $\delta_1, \dots, \delta_\ell$ are non-trivial semidegrees.
          \item[(iv)] The irreducible components $D_1, \dots, D_\ell$ of $D$ are in natural one-to-one correspondence with the semidegrees $\delta_1, \dots, \delta_\ell$, such that
          \[ \mathrm{ord}_{D_i}=-\frac{\delta_i}{d_i}\,,\]
          where $\mathrm{ord}_{D_i}$ is the order of vanishing along $D_i$, and 
          \begin{equation}
\label{eq:di}
    d_i:=\{\gcd(\delta_i(f)) \,|\, f \in R \setminus \{0\}, \,\delta_i(f)>0\} \in \ZZ_{>0}\,.
\end{equation}
          \item[(v)] The $\QQ$-divisor 
          \begin{equation}
\label{eq:edelta}
     E:=\sum_{i=1}^\ell \frac{1}{d_i}D_i
\end{equation} is an ample $\QQ$-Cartier $\QQ$-divisor supported on $D$. Moreover, for every $d \geq 0$, we have \[ F_d=H^0(Y, \cO_Y(d E))\,,\]
      that is, $F_d \subset R=H^0(X,\cO_X)$ is the subspace of regular functions on $X$ with 
      poles of order at most $\frac{d}{d_i}$ along $D_i$ for all $1 \leq i \leq {\ell}$.
      \end{itemize}
\end{proposition}

\begin{proof}
Since $R$ is an integrally closed domain and $\delta$ is a subdegree, (i)
follows from \cite[Proposition 2.2.7-Corollary 2.2.8]{mondal2010towards}.
Moreover, (iii) is \cite[Theorem 4.1, 3.]{mondal2014projective},
and (ii)-(iv) follow from \cite[Corollary 4.4-Proposition 5.2]{mondal2014projective}. The first part of (v) is established in \cite[Lemma 6.2]{mondal2014projective}. Finally, the second part of (v) follows since $f \in F_d$ is equivalent to $\delta(f) \leq d$, that is $\delta_i(f)\leq d$ for all $1 \leq i \leq \ell$, which is equivalent to $\frac{\delta_i(f)}{d_i} \leq \frac{d}{d_i}$, that is $-\mathrm{ord}_{D_i}(f) \leq \frac{d}{d_i}$ for all $1 \leq i \leq \ell$.
\end{proof}

As illustrated in the following example, the construction of projective compactifications of affine varieties from filtrations generalizes the construction of  projective toric varieties as  compactifications of the torus $\mathbb{G}_m^n$, where $\mathbb{G}_m \simeq \kk^*$.
\begin{example}
\label{Example:toric1}
Fix the lattice $M=\ZZ^n$ and denote by
$M_{\RR}=M\otimes_{\ZZ}\RR$ 
the associated $n$-dimensional real vector space. Let $N \coloneqq \Hom(M,\ZZ)$ be the lattice dual to $M$. Consider the algebra of Laurent polynomials in $n$-variables
$R = \kk[M] = \kk[z_1^{\pm 1}, \ldots ,z_n^{\pm 1}]$,
and set 
\[X = \Spec(R)=\Hom(M,\mathbb{G}_m) \simeq \mathbb{G}_m^n \,. \]
Let $P$ be a convex lattice polytope in $M_{\RR}$ containing the origin in its interior. Let $F_1,\ldots, F_{\ell}$ be the facets (that is, the faces of codimension $1$) of $P$. For each $1 \leq i \leq \ell $, there exists a unique primitive $\alpha_i \in N$ and a unique $a_i \in \ZZ_{> 0}$ such that $F_i$ is contained in the affine hyperplane $\alpha_i = a_i$. Here, $\alpha_i \in N$ is the primitive outer normal to the facet $F_i$ and $a_i \in \ZZ_{> 0}$ is the lattice distance from the origin to $F_i$. Let $a = \text{lcm}((a_i)_{1\leq i \leq \ell})$, and for any $1 \leq i \leq \ell$, denote
\begin{align}
\label{eq:deltai}
\delta_i \colon R & \longrightarrow \ZZ \cup \{ -\infty \} \\
\nonumber
\sum_{m \in M} c_m z^m & \longmapsto  \frac{a}{a_i} \max_{ \substack{ ~~m \in M\\~~c_m \neq 0}} \alpha_i(m) \,.
\end{align}
Then, for each $1 \leq i \leq \ell$, $\delta_i$ is a semidegree, and  $\delta \coloneqq \max_{1 \leq i \leq \ell } \delta_i$ is a subdegree defining a projective filtration $\cF= (F_d)_{d \geq 0}$ on $R$. Since $\delta(z^m) = a$ for every $m \in M$ on the boundary of $P$, it follows that 
\begin{equation} \label{Eq:Fd_toric}
F_d = \bigoplus_{m \in M \cap \frac{d}{a}P} \kk z^m\,.\end{equation} 
Hence, $Y$ is the projective toric variety with momentum polytope $P$ and the divisors $D_i$ are the toric divisors. Moreover, let $L$ be the ample line bundle on $Y$ determined by $P$. 
Then, the origin inside $P$ corresponds to a section of $L$ with vanishing locus $\sum_{i=1}^{\ell} a_i D_i$, and so we obtain an isomorphism $L \simeq \mathcal{O}_Y(\sum_{i=1}^{\ell} a_i D_i)$. 
In particular, $\sum_{i=1}^{\ell} a_i D_i$ is a Cartier divisor. 
It follows from \eqref{eq:di} and \eqref{eq:deltai} that $d_i = \frac{a}{a_i}$. Hence, by \eqref{eq:edelta} we obtain 
\[
E = \sum_{i=1}^{\ell} \frac{1}{d_i} D_i =  \sum_{i=1}^{\ell} \frac{a_i}{a} D_i= \frac{1}{a}  \sum_{i=1}^{\ell} a_i D_i\,, \]  
thus $L= \mathcal{O}_Y(aE)$, that is, $(Y, E)$ is the $\QQ$-polarized projective variety defined by the rational polytope 
$\frac{1}{a}P$.
\end{example}

As a special case of the construction of projective compactifications from degree-like functions, we illustrate below how to obtain projective compactifications of affine cones.

\begin{example}
\label{Example:projcone}
  Let $R^0=\bigoplus_{d \geq 0} R_d^0$ be a finitely generated $\ZZ_{\geq 0}$-graded algebra such that $R_0^0 = \kk$. Then, $X_0:=\Spec(R^0)$ is an affine variety with a $\mathbb{G}_m$-action with a unique fixed point $\Spec(R_0^0) \hookrightarrow X_0$. In this case, $X_0$ is the affine cone over the projective variety $D_0:=\Proj(R^0)$. 
  Since $R^0$ is a graded algebra, it is in particular a filtered algebra for the filtration $\cF^0=(F_d^0)_{d\geq 0}$ defined by 
 \[ F_{d}^0 \coloneqq \bigoplus_{0 \leq r \leq d} R_r^0\,.\] 
 The corresponding projective compactification of $X_0$
is $(Y_0,D_0)$, where  \[  Y_0 \coloneq \Proj(\mathrm{Rees}_{\cF}(R^0))= \Proj \left( \bigoplus_{d \geq 0} F_d^0 t^d \right) = \Proj  \left( \bigoplus_{d \geq 0} \bigoplus_{0 \leq r \leq d} R_r^0 t^d \right) \,, \]
 and 
 \[ D_0 = \Proj(\mathrm{gr}_\cF(R^0)) =\Proj(R^0) \,.\]
\end{example}

\subsection{Compactifications of $SL(2,\CC)$ character varieties}
\label{sec_compactifications_SL2}

\subsubsection{The compact case}
\label{sec_review_compact}

In this section, we review the construction in \cite{ayilliath2025geometric} of projective compactifications $(Y_\cP,D_\cP)$ of the character variety $
X:=\mathrm{Ch}_g^{SL(2,\CC)}$ for $g \geq 2$.
The construction is based on the projective compactifications arising from filtrations, reviewed in the previous section
\S\ref{sec_filtrations}. We begin by recalling that $\mathrm{Ch}_g^{SL(2,\CC)}$ is a normal irreducible affine variety:

\begin{proposition} \label{prop_normal_closed}
    For every $g \geq 2$, the $SL(2,\CC)$-character variety $X=\mathrm{Ch}_g^{SL(2,\CC)}$ is a normal irreducible affine variety of dimension $6g-6$.
\end{proposition}

\begin{proof}
More generally, \cite[\S 2.2]{bellamy2023symplectic} shows that, for every $n \geq 1$, the $SL(n,\CC)$-character variety $\mathrm{Ch}_g^{SL(n,\CC)}$ is a normal irreducible variety of dimension $(2g-2)(n^2-1)$. The proof reduces via \cite[Lemma 2.3]{bellamy2023symplectic} 
to the corresponding result for $GL(n,\CC)$, namely that, for every $n \geq 1$, the $GL(n,\CC)$-character variety $\mathrm{Ch}_g^{GL(n,\CC)}$ is a normal irreducible variety of dimension $(2g-2)n^2+2$, established earlier in \cite[Theorem~11.1]{simpson1994moduliII}.
\end{proof}

A key ingredient in the construction of \cite{ayilliath2025geometric} is the following description \eqref{eq_R_basis} of a basis for the algebra of regular functions $R:=H^0(X,\mathcal O_X)$. We first introduce the topological notions required to formulate this description.
Fix a spin structure on $\bS_g$. By \cite{johson1980spin}, it determines a quadratic refinement 
\[ \sigma: H_1(\bS_g, \ZZ/2\ZZ) \rightarrow \ZZ/2\ZZ\] 
of the intersection product on $H_1(\bS_g, \ZZ/2\ZZ)$, namely a map such that
\[ \sigma(\alpha+\alpha')=\sigma(\alpha)+\sigma(\alpha')+\alpha \cdot \alpha'\]
for all $\alpha, \alpha' \in H_1(\bS_g, \ZZ/2\ZZ)$.
We briefly recall the definition of $\sigma(\alpha)$; for a detailed treatment, see \cite{johson1980spin}. Fix an orientation and a Riemannian metric on $\bS_g$, and represent $\alpha$ by a disjoint union of simple closed curves $\alpha_i$. For each component $\alpha_i$, choose an orientation and let $T_i$ denote the corresponding unit tangent vector field. Let $N_i$ be the unique unit normal vector field such that $(T_i,N_i)$ is a positively oriented orthonormal frame along $\alpha_i$. This frame determines a loop in the oriented orthonormal frame bundle of $\bS_g$. The value $\sigma(\alpha_i)\in\mathbb{Z}/2\mathbb{Z}$ is defined as the obstruction to lifting this loop to the spin bundle. Finally, define
\[\sigma(\alpha):=\sum_i \sigma(\alpha_i)\,.\]

Let $SCC_g$ be the set of isotopy classes of simple closed curves, which are non-trivial in the sense that they do not bound a disk in $\bS_g$, and let $MC_g$ be the set of isotopy classes of multicurves in $\bS_g$, that is, of (possibly empty) finite disjoint unions of non-trivial simple closed curves in $\bS_g$. 
We denote by $[-]: SCC_g \rightarrow H_1(\bS_g, \ZZ/2\ZZ)$ the map sending an isotopy class of simple closed curves to its $\ZZ/2\ZZ$-homology class.
For every $\gamma \in SCC_g$, let $T_\gamma \in R=H^0(X,\cO_X)$ be the regular function sending a representation $\rho: \pi_1(\bS_g) \rightarrow SL(2,\CC)$ to 
\begin{equation}
\label{eq_spin}
T_{\gamma}(\rho):=  - (-1)^{\sigma([\gamma])} \mathrm{tr}(\rho(\gamma)) \,.\end{equation}
For every $\gamma \in MC_g$, with connected components $\gamma_1, \dots, \gamma_m$, we define $T_\gamma \in R$ by 
\begin{equation} 
\label{eq_T_gamma}
T_\gamma(\rho):= \prod_{i=1}^m T_{\gamma_i} \,.\end{equation}
It follows from \cite{Charles2012multicurves} that  $(T_\gamma)_{\gamma \in MC_g}$ is a $\kk$-linear basis of $R$, that is, 
\begin{equation} \label{eq_R_basis}
R = \bigoplus_{\gamma \in MC_g} \kk T_\gamma \,.\end{equation}
\begin{remark}
The basis of $R$ considered in \cite{ayilliath2025geometric} is obtained from the basis of multicurves in the Kauffman bracket skein algebra by specializing to \(q=1\), after identifying the resulting algebra with \(R\) via a choice of spin structure. On the other hand, \cite{Charles2012multicurves} establishes a canonical, spin-structure-independent identification between \(R\) and the specialization of the skein algebra at \(q=-1\), sending a simple closed curve \(\gamma\) to the function \(-\mathrm{tr}(\rho(\gamma))\). By \cite{barrett1999skein}, twisting by \((-1)^{\sigma([\gamma])}\) is precisely what relates this canonical identification to the one arising at \(q=1\) from a chosen spin structure. Consequently, the basis \((T_\gamma)_{\gamma \in MC_g}\) defined above coincides with the basis introduced in \cite{ayilliath2025geometric}.
\end{remark}

To study the discrete countable set $MC_g$ indexing the basis $(T_\gamma)_{\gamma \in MC_g}$, it is convenient to embed it into the 
larger continuous space  $MF_g$ of (Whitehead equivalence classes of) measured foliations on $\bS_g$. We refer to \cite{thurston} for background on measured foliations. The set $MF_g$ carries a natural topology, and there is a natural discrete inclusion $MC_g \subset MF_g$. Moreover, by \cite[Theorem 6.15]{thurston}, $MF_g$ carries a scaling action of $\RR_{>0}$, with unique fixed point $0$ given by the empty multicurve, and the quotient $(MF_g  \setminus \{0\})/\RR_{>0}$ is homeomorphic to the sphere of dimension $6g-7$. 
 The geometric intersection number $i(-,-)$ between isotopy classes of simple closed curves, that is, the minimal number of intersection points between curves in these isotopy classes, extends to multicurves, and then to measured foliations to produce a pairing 
 \[ i(-,-): MF_g \times MF_g \longrightarrow \RR_{\geq 0}\,.\]

Following \cite{ayilliath2025geometric}, we now review the definition of a filtration $\cF_\cP$ on $R$ for every choice
$\cP=(P,\Gamma)$ of a pair of pants decomposition $P$ of $\bS_g$ and of an embedded 3-valent graph $\Gamma \subset \bS_g$ dual to $P$.
As described in detail in \cite[\S 3]{ayilliath2025geometric}, these data determine $3g-3$ curves $a_1, \dots, a_{3g-3}$, which are the 
boundaries of the pairs of pants, $3g-3$ curves $a_1', \dots, a_{3g-3}'$ which are dual to $a_1,\dots, a_{3g-3}$, and $3g-3$ curves $a_1'', \dots, a_{3g-3}''$, where $a_j''$ is the $+1$ Dehn twist of $a_j'$ along $a_j$. We denote by $\cC=(c_j)_{1 \leq j \leq 9g-9}$ the set of these $9g-9$ curves. Consider the map
\begin{align}\label{eq_iota_P}
\iota_\cP: MF_g &\longrightarrow \RR_{\geq 0}^{9g-9} 
\\  \nonumber
\gamma &\longmapsto (i(c_j,\gamma))_{1 \leq j \leq 9g-9} \,.
\end{align}
By the $(9g-9)$-Theorem, \cite[Theorem 4.10]{thurston}-\cite[p438]{farb_margalit}, reviewed in \cite[Theorem 1.4]{ayilliath2025geometric}, the map $\iota_\cP$ is injective, and 
its image $\iota_\cP(MF_g) \subset \RR_{\geq 0}^{9g-9}$ is the union $\bigcup_{\ell} \Lambda_\ell$ of finitely many $(6g-6)$-dimensional rational polyhedral cones $\Lambda_\ell$ in $\RR_{\geq 0}^{9g-9}$ glued along their faces. Requiring this decomposition to be the coarsest possible uniquely determines the cones. 
Since $MF_g$ is homeomorphic to a cone over the sphere of dimension $6g-7$, it follows that $\bigcup_\ell \Lambda_\ell$ is likewise homeomorphic to the cone over the sphere of dimension $6g-7$.
Moreover, for every $\ell$, the set
$Q_\ell := \iota_\cP(MC_g) \cap \Lambda_\ell$ is a finite index saturated submonoid of $\Lambda_\ell \cap \ZZ_{\geq 0}^{9g-9}$.

By \cite[Theorem 1.5]{ayilliath2025geometric}, the map
\begin{align} \label{eq_delta_P}
\delta_\cP: R &\longrightarrow \ZZ_{\geq 0}\\   \nonumber
\sum_{\gamma \in MC_{g}} a_\gamma T_\gamma &\longmapsto \max_{\substack{\gamma \in MC_g \\ a
_\gamma \neq 0}} \left( \sum_{j=1}^{9g-9} i(c_j, \gamma) \right) \,.
\end{align}  is a degree-like function on $R$, and in fact a subdegree. Let $\cF_\cP=(F_{\cP,d})_{d\geq 0}$ be the corresponding non-negative filtration on $R$. 
The injectivity of $\iota_\cP$ implies that $F_{\cP,0}=\kk$.
Finally, by \cite[Theorem 4.2]{ayilliath2025geometric}, the associated graded algebra $\gr_{\cF_\cP}(R)$ has the following explicit description: for every $\gamma \in MC_g$, denote by $\overline{T}_\gamma$ the image of $T_\gamma$ in $F_{\cP, \delta_\cP(\gamma)}/F_{\cP, \delta_\cP(\gamma)-1} \subset \gr_{\cF_\cP}(R)$. Then, $(\overline{T}_\gamma)_{\gamma \in MC_g}$ is a basis of $\gr_{\cF_\cP}(R)$, that is,
\[ \gr_{\cF_\cP}(R) = \bigoplus_{\gamma \in MC_g} \kk \overline{T}_\gamma \,,\]
and the product is explicitly given by:
\begin{equation} \label{eq_algebra_compact}
\overline{T}_\gamma \overline{T}_{\gamma'}
=\begin{cases} \overline{T}_{\iota_{\cP}^{-1}(\iota_\cP(\gamma)+ \iota_\cP(\gamma'))} &\text{if there exists  }\ell \text{   such that }\iota_\cP(\gamma), \iota_{\cP}(\gamma') \in \Lambda_\ell\\ 0&\text{else}\,.\end{cases}
\end{equation}
Since the cones $\Lambda_\ell$ are rational polyhedral, and $Q_\ell=\iota_\cP(MC_g) \cap \Lambda_\ell$
are finite index saturated submonoids of $\Lambda_\ell \cap \ZZ_{\geq 0}^{9g-9}$, the monoids $Q_\ell$ are finitely generated by Gordan's Lemma (see for instance \cite[Proposition 1.2.17]{CLS}). Therefore, $
\gr_{\cF_{\cP}}(R)$ is a finitely generated $\kk$-algebra.
By Proposition \ref{prop_finitely_generated}, it follows that the filtration $\cF_\cP$ is finitely generated, and hence is a projective filtration on $R$.
Since $X$ is normal by Proposition \ref{prop_normal_closed}, we obtain by Proposition \ref{prop_proj_comp} 
a normal projective variety $Y_\cP$, and a reduced effective Weil divisor $D_\cP$ on $Y_\cP$ such that $X=Y_\cP \setminus D_\cP$. 

\subsubsection{The punctured case} 
\label{sec_review_punctured}
Let $g \geq 0$, $n>0$ such that $2g-2+n>0$, and $t \in \CC^n$. In this section, we review the construction in 
\cite{farajzadeh2023compactifications} of projective compactifications $(Y_\Delta, D_\Delta)$ of the relative character variety $X:= \mathrm{Ch}_{g,n}^{SL(2,\CC),t}$. We begin by recalling that $\mathrm{Ch}_{g,n}^{SL(2,\CC),t}$ is a normal irreducible affine variety: 

\begin{proposition} \label{prop_normal_punctured}
    For every $g \geq 0$, $n>0$ such that $2g-2+n>0$, and every $t \in \CC^n$, the relative $SL(2,\CC)$-character variety $X=\mathrm{Ch}_{g,n}^{SL(2,\CC), t}$ is a normal irreducible affine variety of dimension $6g-6+2n$.
\end{proposition}

\begin{proof}
This is proved in \cite[Theorem 1.1]{MR4157427}.
\end{proof}

As in the previous section \S\ref{sec_review_compact}, we denote by $MC_{g,n}$ the set of isotopy classes of multicurves in $\bS_{g,n}$, and by $MF_{g,n}$ the space of measured foliations on $\bS_{g,n}$. Let 
$MC_{g,n}^{\mathrm{ess}} \subset MC_{g,n}$ be the subset of essential multicurves, 
that is, of multicurves none of whose connected components are peripheral, and let $MF_{g,n}^{\mathrm{ess}} \subset MF_{g,n}$ be the corresponding subspace of essential measured foliations. By \cite[Theorem 11.1]{thurston}, $MF_{g,n}^{\mathrm{ess}}$ is homeomorphic to a cone over the sphere of dimension $6g-7+2n$.

We fix a spin structure on $\bS_{g,n}$, and define for every $\gamma \in MC_{g,n}^{\mathrm{ess}}$ a regular function $T_\gamma \in R=H^0(X,\cO_X)$ by formulas \eqref{eq_spin}-\eqref{eq_T_gamma}. Then, $(T_\gamma)_{\gamma \in MC_{g,n}^{\mathrm{ess}}}$ is a $\kk$-linear basis of $R$, that is, 
\[ R = \bigoplus_{\gamma \in MC_{g,n}^{\mathrm{ess}}} \kk T_\gamma \,.\]

Let $\Delta$ be an ideal triangulation of $\bS_{g,n}$, that is, a triangulation of the closed surface $\bS_g$ such that the set of vertices is the set of
punctured points. The sets $V(\Delta)$, $E(\Delta)$, $T(\Delta)$  of vertices, edges, and triangles of $\Delta$ respectively have cardinalities $|V(\Delta)|=n$, 
$|E(\Delta)|=3(n+2g-2)=6g-6+3n$, and $|T(\Delta)|=2(n+2g-2)=4g-4+2n$.
Denote by $i(-,-)$ the geometric intersection number, and consider the map
\begin{align} \label{eq_iota_Delta}
\iota_\Delta: MF_{g,n}^{\mathrm{ess}} &\longrightarrow \RR_{\geq 0}^{E(\Delta)} \simeq \RR_{\geq 0}^{6g-6+3n} 
\\ \nonumber
\gamma &\longmapsto (i(c,\gamma))_{c\in E(\Delta)} \,.
\end{align}
As reviewed in \cite[\S 2.3]{farajzadeh2023compactifications}, the map $\iota_\Delta$ is injective, and its image $\iota_\Delta(MF_{g,n}^{\mathrm{ess}}) \subset \RR_{\geq 0}^{E(\Delta)}$ is 
the union $\bigcup_{\ell} \Lambda_\ell$ of finitely many $(6g-6+2n)$-dimensional rational polyhedral cones $\Lambda_\ell$ in $\RR_{\geq 0}^{|E(\Delta)|}$ glued along their faces. Requiring this decomposition to be the coarsest possible uniquely determines the cones. Since $MF_{g,n}^{\mathrm{ess}}$ is homeomorphic to a cone over the sphere of dimension $6g-7+2n$, it follows that $\bigcup_\ell \Lambda_\ell$ is likewise homeomorphic to the cone over the sphere of dimension $6g-7+2n$. Moreover, for every $\ell$, the image 
$Q_\ell := \iota_\Delta(MC_{g,n}^{\mathrm{ess}}) \cap \Lambda_\ell$ is a finite index saturated submonoid of $\Lambda_\ell \cap \ZZ_{\geq 0}^{E(\Delta)}$.

By \cite[\S 2.4]{farajzadeh2023compactifications}, the map
\begin{align} \label{eq_delta_Delta} \delta_\Delta: R &\longrightarrow \ZZ_{\geq 0}\\
\sum_{\gamma \in MC_{g,n}^{\mathrm{ess}}} a_\gamma T_\gamma &\longmapsto \max_{\substack{\gamma \in MC_{g,n}^{\mathrm{ess}} \\ \nonumber
a_\gamma \neq 0}} \left( \sum_{c \in E(\Delta)} i(c, \gamma) \right) \,.
\end{align}
is a degree-like function on $R$, and in fact a subdegree. Let $\cF_\Delta=(F_{\Delta,d})_{d\geq 0}$ be the corresponding non-negative filtration on $R$. Moreover, $F_{\Delta,0}=\kk$ by the injectivity of $\iota_\Delta$. 
Finally, by \cite[\S 2.4]{farajzadeh2023compactifications}, the associated graded algebra $\gr_{\cF_\Delta}(R)$ has the following explicit description: for every $\gamma \in MC_{g,n}^{\mathrm{ess}}$, denote by $\overline{T}_\gamma$ the image of $T_\gamma$ in $F_{\Delta, \delta_\Delta(\gamma)}/F_{\Delta, \delta_\Delta(\gamma)-1} \subset \gr_{\cF_\Delta}(R)$. Then, $(\overline{T}_\gamma)_{\gamma \in MC_{g,n}^{\mathrm{ess}}}$ is a basis of $\gr_{\cF_\Delta}(R)$, that is,
\[ \gr_{\cF_\Delta}(R) = \bigoplus_{\gamma \in MC_{g,n}^{\mathrm{ess}}} \kk \overline{T}_\gamma \,,\]
and the product is explicitly given by:
\begin{equation} \label{eq_algebra_punctured}
\overline{T}_\gamma \overline{T}_{\gamma'}
=\begin{cases} \overline{T}_{\iota_{\Delta}^{-1}(\iota_\Delta(\gamma)+ \iota_\Delta(\gamma'))} &\text{if there exists  }\ell \text{   such that }\iota_\Delta(\gamma), \iota_\Delta(\gamma') \in \Lambda_\ell\\ 0&\text{else}\,.\end{cases}
\end{equation}
Since the cones $\Lambda_\ell$ are rational polyhedral, and $Q_\ell=\iota_\Delta(MC_{g,n}^{
\mathrm{ess}}) \cap \Lambda_\ell$
are finite index saturated submonoids of $\Lambda_\ell \cap \ZZ_{\geq 0}^{E(\Delta)}$, the monoids $Q_\ell$ are finitely generated by Gordan's Lemma (see for instance \cite[Proposition 1.2.17]{CLS}). Thus, $
\gr_{\cF_\Delta}(R)$ is a finitely generated $\kk$-algebra. 
By Proposition \ref{prop_finitely_generated}, this implies that the filtration $\cF_\Delta$ is finitely generated, and hence is a projective filtration on $R$.
Since $X$ is normal by Proposition \ref{prop_normal_punctured}, Proposition \ref{prop_proj_comp} produces 
a normal projective variety $Y_\Delta$, and a reduced effective Weil divisor $D_\Delta$ on $Y_\Delta$ such that $X=Y_\Delta \setminus D_\Delta$.

\section{Log Calabi--Yau compactifications from filtrations}
\label{sec_log_cy}

In \S\ref{sec_algebras_cone}, we review the algebras associated with cone complexes, 
obtained by gluing monoid algebras. In \S\ref{sec_lcy_statement}, we state Theorem \ref{thm_main_C} (Theorem~\ref{thm_intro_C}), which provides a general criterion for constructing log Calabi--Yau compactifications from filtrations. In \S\ref{sec_lcy_character}, we apply this criterion to the compactifications of the $SL(2,\CC)$ character varieties reviewed in \S\ref{sec_compactifications_SL2}, to prove Theorems
\ref{thm_main_B}--\ref{thm_main_A} (Theorems ~\ref{thm_intro_A}--\ref{thm_intro_B}). 
The proof of Theorem \ref{thm_main_C} (Theorem~\ref{thm_intro_C}) is given in \S\ref{section_proof}.

\subsection{Algebras of cone complexes}
\label{sec_algebras_cone}

Roughly speaking, a cone complex is a finite collection of rational polyhedral cones glued along their faces by integral linear maps without self-intersection. We will use the following definition -- see  \cite[II1, Definition 5, p69]{KKFMSD} and \cite[Definition 2.1]{payne2009toric}.

\begin{definition}
A \emph{cone complex} $\Omega$ consists of a topological space $|\Omega|$, together with a finite collection of closed subspaces of $|\Omega|$, called the \emph{cones} of $\Omega$, and, for each cone $\omega$, a finitely generated group $N_\omega$ of continuous real-valued functions on $\omega$, called the \emph{integral linear functions} on $\omega$, such that, writing
$ M_\omega:=\Hom(N_\omega,\ZZ)$, 
the following conditions are satisfied:
\begin{itemize}
    \item[(i)] The evaluation map
    \[
    \phi_\omega:\omega \longrightarrow M_\omega \otimes \RR,
    \qquad
    x\longmapsto \bigl(f\mapsto f(x)\bigr),
    \]
    is a homeomorphism from $\omega$ onto a rational polyhedral cone in the real vector space $M_\omega \otimes \RR$ endowed with the lattice $M_\omega \subset M_\omega \otimes\RR$.
    \item[(ii)] For every face $\tau$ of $\phi_\omega(\omega)$, the preimage
$\rho:=\phi_\omega^{-1}(\tau)$
    is a cone of $\Omega$, and
\[ N_\rho=\{\,f|_\rho \mid f\in N_\omega\,\}\,.\]
    \item[(iii)] The underlying topological space $|\Omega|$ is the disjoint union of the relative interiors of the cones of $\Omega$.
\end{itemize}
\end{definition}

\begin{definition}
    Let $\Omega$ be a cone complex. A point $x \in 
    |\Omega|$ is \emph{integral} if there exists a cone $\omega$ of $\Omega$ such that $x$ is an integral point of $\omega$, that is,  $\phi_\omega(x) \in M_\omega$. 
    We denote the set of integral points of $\Omega$ by $\Omega(\ZZ)$.
\end{definition}

\begin{definition}
Let $\Omega$ be a cone complex. An integral piecewise linear function $h$ on $\Omega$ is a continuous function $h: |\Omega| \rightarrow \RR$ whose restriction to every cone $\omega$ of $\Omega$
 is an integral linear function, that is, $h|_\omega \in N_\omega$.
\end{definition}

\begin{remark}
    An integral piecewise linear function $h: |\Omega| \rightarrow \RR$ takes integer values on the set of integral points, that is,  $h(\Omega(\ZZ)) \subset \ZZ$. 
\end{remark}

\begin{definition} \label{def_algebra_cone}
The \emph{algebra of a cone complex $\Omega$} is the $\kk$-algebra $\kk[\Omega]$ with underlying $\kk$-vector space 
\[ \kk[\Omega]= \bigoplus_{m \in \Omega(\ZZ)} \kk z^m \,, \]
and with product given on basis elements by:
\[ z^m z^{m'}=\begin{cases}
    z^{\phi_\omega^{-1}(\phi_\omega(m)+\phi_\omega(m'))} & \text{if there exists a cone } \omega \text{ of }\Omega\text{ such that }m,m' \in \omega \\
    0 & \text{else.}
\end{cases}\]
\end{definition}

\begin{remark}
When $\Omega$ is a simplicial complex, the associated algebra $\kk[\Omega]$ is known as the Stanley--Reisner algebra of $\Omega$; see \cite[Chapter II]{stanley_book}.
\end{remark}

A non-negative integral piecewise linear function $h: |\Omega| \rightarrow \RR_{\geq 0}$ on a cone complex $\Omega$ defines a $\ZZ_{\geq 0}$-grading $\deg_h$ on the algebra $\kk[\Omega]$, given by
\[ \deg_h(z^m):= h(m) \,.\]

\subsection{Statement of the log Calabi--Yau criterion}
\label{sec_lcy_statement}

The following theorem is the main technical result of this paper and is stated as Theorem~\ref{thm_intro_C} in the introduction. It gives a sufficient criterion for projective compactifications arising from filtrations to be log canonical and log Calabi--Yau. The crucial assumption is that the associated graded algebra is the algebra of a cone complex, in the sense of \S\ref{sec_algebras_cone}, homeomorphic to the cone over a sphere.

\begin{theorem}\label{thm_main_C}
Let $\cF=(F_d)_{d\geq 0}$ be a projective filtration on a $\kk$-algebra $R$, and let $(Y,D)$ be the associated projective compactification of $X=\Spec(R)$. Assume that:
\begin{itemize}
\item[(i)] $R$ is an integrally closed domain, equivalently, $X$ is a normal variety.
\item[(ii)] The degree-like function $\delta$ defining $\cF$ is a subdegree.
\item[(iii)] There exist a cone complex $\Omega$ and a non-negative integral piecewise linear function $ h:|\Omega|\longrightarrow \RR_{\geq 0}$ such that the intersection of any two cones of $\Omega$ is a cone of $\Omega$, the level set $h^{-1}(1)$ is homeomorphic to a sphere, and the algebra $\kk[\Omega]$, endowed with the grading induced by $h$, is isomorphic to $\gr_{\cF}(R)$ as a $\ZZ_{\geq 0}$-graded algebra.
\end{itemize}
Then:
\begin{itemize}
\item[(i)] $Y$ is a normal projective variety and $D$ is a reduced effective Weil divisor on $Y$,
\item[(ii)] $Y$ is Cohen--Macaulay and $(Y,D)$ is log canonical, and
\item[(iii)] $(Y,D)$ is log Calabi--Yau, that is, $K_Y+D\sim 0$.
\end{itemize}
\end{theorem}

The proof of Theorem~\ref{thm_main_C} is carried out in \S\ref{section_proof}, with the final step completed in \S\ref{section_end_proof}. In the subsequent section  \S\ref{sec_lcy_character}, we apply Theorem~\ref{thm_main_C} to the compactifications of $SL(2,\CC)$ character varieties reviewed in \S \ref{sec_compactifications_SL2}.

\subsection{Log Calabi--Yau compactification of character varieties}
\label{sec_lcy_character}

The following result is stated as Theorem \ref{thm_intro_B} in the introduction. 

\begin{theorem}
\label{thm_main_B}
 The following hold:
\begin{itemize}
\item[(i)] \emph{(Compact case)}
Let $g\geq 2$, and let $\cP=(P,\Gamma)$ consist of a pair of pants decomposition $P$ of $\bS_g$ together with an embedded dual graph $\Gamma$. Then $Y_\cP$ is Cohen--Macaulay, and the compactification $(Y_\cP,D_\cP)$ of $\mathrm{Ch}_g^{SL(2,\CC)}$ is log canonical and log Calabi--Yau.

\item[(ii)] \emph{(Punctured case)}
Let $g\geq 0$ and $n>0$ satisfy $2g-2+n>0$, let $t\in\CC^n$, and let $\Delta$ be an ideal triangulation of $\bS_{g,n}$. Then
 $Y_\Delta$ is Cohen--Macaulay, and the compactification $(Y_\Delta,D_\Delta)$ of $\mathrm{Ch}_{g,n}^{SL(2,\CC),t}$ is log canonical and log Calabi--Yau.
\end{itemize}
\end{theorem}

\begin{proof}
We verify that the assumptions of Theorem \ref{thm_main_C} are satisfied in each of the two cases.

In the compact case, the affine variety $X=\mathrm{Ch}_g^{SL(2,\CC)}$ is normal by Proposition \ref{prop_normal_closed}. 
As reviewed in \S\ref{sec_review_compact}, following \cite{ayilliath2025geometric}, the pair $(Y_\cP, D_\cP)$ is the projective compactification of $X$ associated to a projective filtration $\cF_\cP$ on $R=H^0(X,\cO_X)$ defined by the subdegree $\delta_\cP$ given by \eqref{eq_delta_P}. 
The associated graded algebra $\mathrm{gr}_{\cF_\cP}(R)$, described by  \eqref{eq_algebra_compact},
is isomorphic to the algebra of a cone complex $\Omega_\cP$, given by the union of cones $\bigcup_\ell \Lambda_\ell \subset \RR_{\geq 0}^{9g-9}$ endowed with the integral structure defined by the monoids $Q_\ell=\iota_\cP(MC_g) \cap \Lambda_\ell$. Moreover, \eqref{eq_iota_P}-\eqref{eq_delta_P} show that the grading on $\mathrm{gr}_{\cF_\cP}(R)$
is induced by the sum of coordinates function on $\RR_{\geq 0}^{9g-9}$.
The corresponding level set in $\Omega_\cP$ is homeomorphic to the sphere of dimension $6g-7$. 
Hence, all the assumptions of Theorem \ref{thm_main_C} are satisfied and the compact case of Theorem \ref{thm_main_B} follows.

In the punctured case, the argument is completely analogous. First, the affine variety $X=\mathrm{Ch}_{g,n}^{SL(2,\CC), t}$ is normal by Proposition \ref{prop_normal_punctured}. 
Then, as reviewed in \S\ref{sec_review_punctured}, following \cite{farajzadeh2023compactifications}, the pair $(Y_\Delta, D_\Delta)$ is the projective compactification of $X$ associated to a projective filtration $\cF_\Delta$ on $R=H^0(X,\cO_X)$ defined by the subdegree $\delta_\Delta$ given by \eqref{eq_delta_Delta}. 
The associated graded algebra $\mathrm{gr}_{\cF_\Delta}(R)$, described by  \eqref{eq_algebra_punctured},
is isomorphic to the algebra of a cone complex $\Omega_\Delta$, given by the union of cones $\bigcup_\ell \Lambda_\ell \subset \RR_{\geq 0}^{6g-6+3n}$ endowed with the integral structure defined by the monoids $Q_\ell=\iota_\Delta(MC_{g,n}^{\mathrm{ess}}) \cap \Lambda_\ell$. Moreover, \eqref{eq_iota_Delta}-\eqref{eq_delta_Delta} show that the grading on $\mathrm{gr}_{\cF_\Delta}(R)$
is induced by the sum of coordinates function on $\RR_{\geq 0}^{6g-6+3n}$.
The corresponding level set in $\Omega_\Delta$ is homeomorphic to the sphere of dimension $6g-7+2n$. 
Hence, all the assumptions of Theorem \ref{thm_main_C} are satisfied and the punctured case of Theorem \ref{thm_main_B} follows.
\end{proof}

Finally, we state and prove Theorem~\ref{thm_main_A}, which is Theorem \ref{thm_intro_A} in the introduction. The proof relies on the following general lemma, which is a consequence of the Minimal Model Program and allows one to obtain dlt log Calabi--Yau compactifications from log canonical ones. Related results include \cite[Theorem 3.1]{kollar_lc_du_bois}, \cite[Theorem 10.4]{fujino_fundamental}, and \cite[Lemma 4.3.1]{mauri2022geometric}.

\begin{lemma} \label{lem_dlt_modifications}
Let $X$ be a $\QQ$-factorial normal variety with canonical singularities. Assume that $X$ admits a log canonical log Calabi--Yau compactification. Then, $X$ admits a $\QQ$-factorial dlt log Calabi--Yau compactification.
\end{lemma}

\begin{proof}
Let $(Y,D)$ be a log canonical log Calabi--Yau compactification of $X$. By the existence of minimal dlt modifications, \cite[Theorem 3.1]{kollar_lc_du_bois}--\cite[Theorem 10.4]{fujino_fundamental}, there exists a projective birational morphism $f: \widetilde{Y} \rightarrow Y$ such that $\widetilde{Y}$ is $\QQ$-factorial, and denoting by $\widetilde{D}$ the union of the strict transforms of the components of $D$ and of the exceptional divisors of $f$, the pair $(\widetilde{Y}, \widetilde{D})$ is dlt, and $K_{\widetilde{Y}}+\widetilde{D} = f^\star(K_Y+D) \sim 0$, that is, $(\widetilde{Y}, \widetilde{D})$ is log Calabi--Yau. 
It remains to show that $f$ induces an isomorphism $\widetilde{Y}\setminus \widetilde{D}\simeq X$. Since $X$ is $\QQ$-factorial, the exceptional locus of $f$ over $X$ is of pure codimension one. Moreover, because $K_{\widetilde{Y}}+\widetilde{D}\sim 0$, every exceptional divisor of $f$ has discrepancy $-1$. On the other hand, $X$ is canonical, so every divisor over $X$ has discrepancy $\geq 0$. Therefore $f$ has no exceptional divisors over $X$, and so $f$ is an isomorphism over $X$.
\end{proof}

\begin{theorem} \label{thm_main_A}
Conjecture \ref{conj_main} holds for $G=SL(2,\CC)$. More precisely:
\begin{itemize}
\item[(i)] \emph{(Compact case)}
For every integer $g\geq 0$, the $SL(2,\CC)$-character variety $\mathrm{Ch}_g^{SL(2,\CC)}$ admits a dlt log Calabi--Yau compactification.

\item[(ii)] \emph{(Punctured case)}
For all integers $g\geq 0$ and $n>0$, and every $t\in\CC^n$, the relative $SL(2,\CC)$-character variety $\mathrm{Ch}_{g,n}^{SL(2,\CC),t}$ admits a log canonical log Calabi--Yau compactification. If, moreover, $t$ is Kostov generic, then $\mathrm{Ch}_{g,n}^{SL(2,\CC),t}$ admits a dlt log Calabi--Yau compactification.
\end{itemize}
\end{theorem}

\begin{proof} In the compact case,  the statement is trivial for $g=0$, is proved in
\cite[Theorem E]{mauri2022geometric}  
for $g=1$, and so we can assume $g \geq 2$.
By Theorem \ref{thm_main_B}, both $\mathrm{Ch}_g^{SL(2,\CC)}$ and $\mathrm{Ch}_{g,n}^{SL(2,\CC),t}$
admit log canonical log Calabi--Yau compactifications. By \cite[Corollary 1.2 -- Theorem 1.3]{bellamy2023symplectic}, the character variety $\mathrm{Ch}_g^{SL(2,\CC)}$ is $\QQ$-factorial and has canonical singularities, and so it follows from Lemma \ref{lem_dlt_modifications} that it admits a dlt log Calabi--Yau compactification. 
On the other hand, by \cite[Proposition 4.5]{simpson2016dual}, the relative character variety $\mathrm{Ch}_{g,n}^{SL(2,\CC),t}$ is smooth, and so in particular $\QQ$-factorial with canonical singularities, whenever $t$ is Kostov generic. Therefore, Lemma \ref{lem_dlt_modifications} 
implies that $\mathrm{Ch}_{g,n}^{SL(2,\CC),t}$ admits a dlt log Calabi--Yau 
compactification when $t$ is Kostov generic.
\end{proof}

\section{Proof of the log Calabi--Yau criterion}
\label{section_proof}

In this section, we prove Theorem \ref{thm_main_C} by a degeneration argument.
In \S\ref{section_degeneration}, we construct canonical degenerations associated with compactifications arising from filtrations. In \S\ref{section_vertices}, we study the central fibers of these degenerations for filtrations satisfying the hypotheses of Theorem~\ref{thm_main_C}. Finally, in \S\ref{section_extending}, we show how properties of the central fiber extend to the general fiber, and we complete the proof of Theorem~\ref{thm_main_C} in \S\ref{section_end_proof}.

\subsection{Degenerations from filtrations}
\label{section_degeneration}
In this section we construct one-parameter degenerations of the projective compactifications constructed from filtrations as in \S\ref{sec_filtrations}.

\subsubsection{Algebraic description of the canonical degeneration}

Let $\cF=(F_d)_{d \geq 0}$ be a projective filtration on a $\kk$-algebra $R$. 
Recall from Definition \ref{def_projective_compactification} that the projective compactification of $X=\Spec(R)$ 
associated with $\cF$ is the pair $(Y,D)$, where $Y=\Proj \left( \mathrm{Rees}_\cF(R) \right)$ and $D=\Proj \left( \gr_\cF(R) \right)$.
Consider the affine scheme $\cX :=\Spec \left( \mathrm{Rees}_\cF(R) \right)$.  The $\kk[t]$-algebra structure on 
$\mathrm{Rees}_\cF(R)$ induces a morphism 
\begin{equation} \label{eq_pi}
\pi: \cX \longrightarrow \A^1=\Spec (\kk[t]) \,,\end{equation}
which is flat since $t$ is a non-zero divisor in $\mathrm{Rees}_\cF(R)$. 
The morphism $\pi$ is $\G_m$-equivariant with respect to the $\G_m$-action by scaling  on 
$\A^1$ and the $\G_m$-action on $\cX$ induced by the $\ZZ_{\geq 0}$-grading of $\mathrm{Rees}_\cF(R)$.
In particular, 
\[ \cX|_{\A^1 \setminus \{0\}} \simeq X \times (\A^1 \setminus \{0\})\,.\] 
Moreover, the central fiber of $\cX$ is 
\[ X_0 := \pi^{-1}(0) = \Spec \left( \gr_\cF(R)\right) \,.\]

In Definition \ref{def_canonical_degeneration} below, we define a particular one-parameter 
degeneration of $(Y,D)$ as a fiberwise compactification 
of $\pi: \cX \rightarrow \A^1$.
First, we define a $\ZZ_{\geq 0} \times \ZZ_{\geq 0}$-graded $\kk$-algebra
\begin{equation}
\label{eq_overline_rees}
     \overline{\mathrm{Rees}}_\cF(R) := \bigoplus_{r \geq 0} \bigoplus_{s \geq 0} (F_r \cap F_s) x^r y^s    \subset R[x,y] \,,
\end{equation}
where $x$ and $y$ are formal variables.
Since $\cF=(F_d)_{d\geq 0}$ is an increasing filtration, we have 
$F_r \cap F_s = F_r$ if $r \leq s$ and 
$F_r \cap F_s=F_s$ if $s \leq r$.
We set
\[ \cY := \Proj_s \left( \overline{\mathrm{Rees}}_\cF(R) \right)\,,\]
where $\Proj_s$ denotes $\Proj$ taken with respect to the $s$-grading.  
Since $F_0=\kk$, the $s$-degree $0$ part of $\overline{\mathrm{Rees}}_\cF(R)$ is 
\[ \bigoplus_{r \geq 0}(F_r \cap F_0) x^r= \bigoplus_{r \geq 0} F_0 x^r= \bigoplus_{r \geq 0} \kk\, x^r =\kk[x]\,,\] 
hence we have an induced projective morphism
\begin{equation} \label{eq_rho}
\rho: \cY \longrightarrow \A^1 = \Spec(\kk[x]) \,,\end{equation}
which is flat since $x$ is a non-zero divisor in $\overline{\mathrm{Rees}}_\cF(R)$.
The morphism $\rho$ is $\mathbb{G}_m$-equivariant with respect to the $\mathbb{G}_m$-action by scaling on $\A^1$ and the $\G_m$-action on $\cY$ induced by the $r$-grading on $\overline{\mathrm{Rees}}_\cF(R)$.

\begin{proposition} \label{prop_degeneration}
Let $\cF=(F_d)_{d\geq 0}$ be a projective filtration on a $\kk$-algebra $R$, and $(Y,D)$ the associated projective compactification of $X=\Spec(R)$. Let $\pi:\cX \rightarrow \A^1$ be the degeneration of $X$ given by \eqref{eq_pi}, and $\rho: \cY \rightarrow \A^1$ be as in \eqref{eq_rho}. Then, the following hold:
\begin{itemize}
    \item[(i)] the quotient map 
    \[ \overline{\mathrm{Rees}}_\cF(R) \longrightarrow \bigoplus_{r \geq 0} \bigoplus_{s \geq 0} (F_r \cap F_s)/(F_r \cap F_{s-1}) x^r y^s \]
    defines a closed embedding 
    \begin{equation} \label{eq_cD}
    \cD:= D \times \A^1 \hookrightarrow \cY\,,
    \end{equation}
    \item[(ii)] $\cY \setminus \cD  \simeq \cX$ and $\rho|_{\cY \setminus \cD} = \pi$,
    \item[(iii)] $\cY|_{\A^1 \setminus \{0\}} \simeq Y \times (\A^1 \setminus \{0\})$,
    \item[(iv)] the central fiber $Y_0 := \rho^{-1}(0)$ is isomorphic to the projective compactification of \[ X_0=\Spec(\gr_\cF(R))\] associated to the filtration $\cF^0=(F^0_d)_{d \geq 0}$ on $\gr_\cF(R)$ defined by $F^0_d := \bigoplus_{0 \leq r \leq d} F_r/F_{r-1}$, that is, 
    \[ Y_0 = \Proj \left( \bigoplus_{d \geq 0} \bigoplus_{0 \leq r \leq d} F_r/F_{r-1} \right) \,.\]
\end{itemize}
\end{proposition}

\begin{proof} 
For every $r \leq s-1$, we have $F_r \cap F_s=F_r$, and $F_r \cap F_{s-1} =F_r$, and so 
\[ (F_r \cap F_s)/(F_r \cap F_{s-1})=0
\,.\] 
On the other hand, for $r \geq s$, we have $F_r \cap F_s=F_s$, and $F_r \cap F_{s-1}=F_{s-1}$, and so \[ (F_r \cap F_s)/(F_r \cap F_{s-1})=F_s/F_{s-1}\,.\] 
Hence, we have an isomorphism
\[ \bigoplus_{r \geq 0} \bigoplus_{s \geq 0} (F_r \cap F_s)/(F_r \cap F_{s-1}) x^r y^s\simeq 
\bigoplus_{r \geq 0} \bigoplus_{s \geq 0} 
(F_s/F_{s-1})x^r y^s = (\gr_\cF(R)) \otimes \kk[x] \,.\]
Since $\Spec ( (\gr_\cF(R)) \otimes \kk[x]) = D \times \A^1=\cD$, this proves (i).
Moreover, (ii) follows from $\overline{\mathrm{Rees}}_\cF(R) \otimes \kk[y^\pm] \simeq \mathrm{Rees}_\cF(R) \otimes \kk[y^\pm]$, and similarly (iii) follows from  $\overline{\mathrm{Rees}}_\cF(R) \otimes \kk[x^\pm] \simeq \mathrm{Rees}_\cF(R) \otimes \kk[x^\pm]$.

The central fiber of $Y_0:=\rho^{-1}(0)$ is given by 
\[ Y_0=\Proj_s (\overline{\mathrm{Rees}}_\cF(R)/(x))\,.\]
To calculate $\overline{\mathrm{Rees}}_\cF(R)/(x)$, first note that the $s$-degree part of $\overline{\mathrm{Rees}}_\cF(R)$ is
\[ \bigoplus_{0\leq r \leq s} F_r x^r y^s \oplus \bigoplus_{r>s} F_s x^r y^s \,.\]
If $r \leq s$, then $F_{r-1} x^r y^s \subset F_r x^r y^s$ is the image of $F_{r-1} x^{r-1}y^s$ by multiplication by $x$. For $r>s$, all of $F_s x^r y^s$ is the image of 
$F_s x^{r-1}y^s$ by multiplication by $x$. 
Therefore, 
we have 
\[ \overline{\mathrm{Rees}}_\cF(R)/(x) \simeq  \bigoplus_{s \geq 0} \bigoplus_{0 \leq r \leq s} F_r/F_{r-1} y^s \simeq \mathrm{Rees}_{\cF^0}(\gr_\cF(R)) \,,\]
and so 
\[ Y_0 =\Proj \left( \mathrm{Rees}_{\cF^0}(\gr_\cF(R))\right) \,,\]
which proves (iv).
\end{proof}

\begin{definition} \label{def_canonical_degeneration}
Let $\cF$ be a projective filtration on a $\kk$-algebra $R$, and let $(Y,D)$ be the associated projective compactification of $X=\Spec(R)$.
The \emph{canonical degeneration} of $(Y,D)$ is the 
$\G_m$-equivariant flat projective morphism
\[\rho: (\cY, \cD) \longrightarrow \A^1\,,\]
where $\rho: \cY \rightarrow \A^1$ is described in \eqref{eq_rho}
and $\cD$ is the constant family $D \times \A^1$ embedded in $\cY$ as in \eqref{eq_cD}.
\end{definition}

\begin{example}
\label{Example:toric2}
In the context of toric geometry, the canonical degeneration of Definition \ref{def_canonical_degeneration} can be viewed as an example of Mumford toric degenerations of toric varieties \cite{Mumford} -- see for instance \cite[Example 3.6]{gross} for an exposition.
Indeed, following Example \ref{Example:toric1}, let $P$ be a convex lattice polytope in $M_\RR$ containing the origin in its interior, $\cF$ the corresponding filtration on $R=\kk[M]\simeq \kk[z_1^\pm, \dots, z_n^\pm]$, and $(Y,D)$ the associated projective compactification of $X=\Spec(R)=\G_m^n$, given by the toric variety with momentum polytope $P$. 
We explain below how to view the canonical degeneration $\rho: (\cY, \cD) \rightarrow \A^1$ of $(Y,D)$
as an example of Mumford toric degeneration.

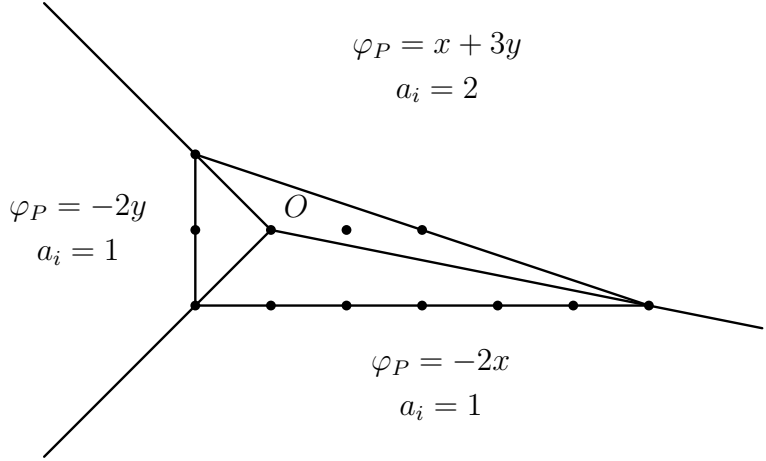
\begin{figure}[htbp]
\centering
\begin{tikzpicture}[x=1.00cm,y=1.00cm,line cap=round,line join=round]
  % Key points
  \coordinate (A) at (-1,1);
  \coordinate (B) at (-1,-1);
  \coordinate (C) at (5,-1);
  \coordinate (O) at (0,0);

  % Half-lines through the three vertices
  \draw[line width=0.9pt] (O) -- ($(O)!3.0!(A)$);
  \draw[line width=0.9pt] (O) -- ($(O)!3.0!(B)$);
  \draw[line width=0.9pt] (O) -- ($(O)!1.3!(C)$);

  % Triangle
  \draw[line width=0.9pt] (A) -- (B) -- (C) -- cycle;

  % Lattice points
  \fill (A) circle (0.065);
  \fill (-1,0) circle (0.065);
  \fill (B) circle (0.065);
  \fill (O) circle (0.065);
  \fill (0,-1) circle (0.065);
  \fill (1,0) circle (0.065);
  \fill (1,-1) circle (0.065);
  \fill (2,0) circle (0.065);
  \fill (2,-1) circle (0.065);
  \fill (3,-1) circle (0.065);
  \fill (4,-1) circle (0.065);
  \fill (C) circle (0.065);

  % Coordinate labels
 
  \node[above right=1pt] at (O) {$O$};

  % Cone labels
  \node[align=center] at (2.2,2.15) {$\varphi_P = x + 3y$\\[1mm]$a_i = 2$};
  \node[align=center] at (-2.55,0.0) {$\varphi_P = -2y$\\[1mm]$a_i = 1$};
  \node[align=center] at (2.25,-2.05) {$\varphi_P = -2x$\\[1mm]$a_i = 1$};
\end{tikzpicture}
\centering
\caption{A convex polytope $P$ containing the origin in its interior, the corresponding polyhedral decomposition $\sP_P$ and the convex piecewise linear function $\varphi_P$. In this example, $(a_i)_i=(1,1,2)$ and $a=2$.}
\label{figure1}
\end{figure}

In general, a Mumford toric degeneration is determined by a convex lattice polytope $P$, an integral polyhedral decomposition $\sP$ of $P$, and a convex piecewise linear function on $P$ whose domains of linearity are the cells of $\sP$. 
Using the notation of Example \ref{Example:toric1}, we consider the piecewise linear function 
\begin{align*}
    \varphi_P: M_\RR &\longrightarrow \RR \\
    m &\longmapsto \max_{1 \leq i \leq \ell} \left( \frac{a}{a_i} \alpha_i(m) \right)\,.
\end{align*}
The polyhedral decomposition $\sP_P$ of $P$ defined by the domains of linearity of $\varphi_P$ restricted to $P$ is the conical star subdivision around $0$, whose maximal dimensional cells are cones centered at $0$ over the facets of $P$ --- see Figure \ref{figure1}. 
The total space of the Mumford toric degeneration associated to $(\frac{1}{a}P, \frac{1}{a}\sP, \varphi_P|_{\frac{1}{a}P})$ is the $\QQ$-polarized toric variety associated to the unbounded rational polytope
    \[ \widetilde{P}=\{ (m,r) \in M_\RR \oplus \RR\,|\, m \in \frac{1}{a}P, \, r \geq \varphi_P(m)\} \,,\]
    that is, 
    the Proj of 
    \[ \bigoplus_{s \geq 0} \bigoplus_{(m,r) \in (s\widetilde{P}) \cap (M \oplus \ZZ) } \kk z^{(m,r,s)} \,. \]
Denote $x=z^{(0,1,0)}$ and $y=z^{(0,0,1)}$. Then the coefficient of $x^r y^s$ has a basis consisting of $z^m$, indexed by
$m \in M$ such that $r \geq \varphi_P(m)$ and $m \in \frac{s}{a}P$, that is $s \geq \varphi_P(m)$.

On the other hand, the definition of $\cF=(F_d)_{d\geq 0}$ in \eqref{Eq:Fd_toric} can be rewritten in terms of $\varphi_P$ as 
    \[F_d = \bigoplus_{\substack{m \in M \\ \varphi_P(m) \leq d }} \kk z^m \,.\]
    Hence, we have 
    \[ \mathrm{Rees}_\cF(R) = \bigoplus_{d \geq 0} \bigoplus_{\substack{m \in M \\ \varphi_P(m) \leq d }} \kk z^m  t^d\]
and 
\[ \overline{\mathrm{Rees}}_\cF(R) = \bigoplus_{r,s \geq 0} \bigoplus_{\substack{m \in M \\ \varphi_P(m) \leq r \\ \varphi_P(m)\leq s}} \kk z^m x^r y^s \,.\]
Therefore, the canonical degeneration $\rho: \cY = \Proj_s( \overline{\mathrm{Rees}}_\cF(R)) \rightarrow \A^1$ is exactly the  
Mumford toric degeneration associated to $(\frac{1}{a}P, \frac{1}{a}\sP, \varphi_P|_{\frac{1}{a}P})$.
\end{example}

\subsubsection{Geometric description of the canonical degeneration: base change}
\label{sec_base_change}
In this section, we provide a geometric description of the canonical degeneration $\rho:(\cY, \cD) \rightarrow \A^1$ introduced in Definition \ref{def_canonical_degeneration}. To do this, we first express the $\ZZ_{\geq 0} \times \ZZ_{\geq 0}$-graded algebra $\overline{\mathrm{Rees}}_\cF(R)$ 
as a base change of the $\ZZ_{\geq 0}$-graded Rees algebra $\mathrm{Rees}_\cF(R)$. 
Recall from \eqref{eq_rees} that $\mathrm{Rees}_\cF(R)$ is naturally a $\kk[t]$-algebra,
and from \eqref{eq_overline_rees} that $\overline{\mathrm{Rees}}_\cF(R)$ is naturally a $\kk[x,y]$-algebra. Viewing $\kk[x,y]$ as a $\kk[t]$-algebra via the ring homomorphism
\begin{align}
\label{eq:t=xy}
    \kk[t] & \longrightarrow \kk[x,y] \\
    \nonumber
    t &\longmapsto xy \,,
\end{align}
we obtain the following lemma.

\begin{lemma}
\label{lem:isomprhismAtilde}
There is a natural isomorphism of $\kk[x,y]$-algebras
\[ \overline{\mathrm{Rees}}_\cF(R) \simeq \mathrm{Rees}_\cF(R) \otimes_{\kk[t]} \kk[x,y] \,.\]
\end{lemma}

\begin{proof}
By \eqref{eq:t=xy}, the relation
$t\otimes 1 = 1\otimes xy$
holds in the tensor product
$\mathrm{Rees}_{\mathcal F}(R)\otimes_{\kk[t]}\kk[x,y]$.
Consequently, for every \(f\in F_d\) and every \(a,b\ge 0\),
\[
f t^d\otimes x^a y^b
 = f\otimes (xy)^d x^a y^b
 = f\otimes x^{a+d}y^{b+d}.
\]
It follows that
\[
\mathrm{Rees}_{\mathcal F}(R)\otimes_{\kk[t]}\kk[x,y]
\simeq
\bigoplus_{r,s\ge0} V_{r,s}\otimes x^r y^s\,,
\]
where
\[
V_{r,s}
\coloneqq
\bigcup_{0\le d\le \min\{r,s\}} F_d\,.
\]
Since \((F_d)_{d\ge0}\) is an increasing filtration,
\[
V_{r,s}
=
F_{\min\{r,s\}}
=
\begin{cases}
F_r, & \text{if } r\le s\,,\\
F_s, & \text{if } s\le r\,.
\end{cases}
\]
Equivalently,
\[
V_{r,s}=F_r\cap F_s\,,
\]
and so the desired conclusion follows from the definition of
\(\overline{\mathrm{Rees}}_{\mathcal F}(R)\) in \eqref{eq_overline_rees}.
\end{proof}

The projective compactification $Y=\Proj(\mathrm{Rees}_\cF(R))$ can be described geometrically as the 
quotient by $\G_{m,t}=\Spec(\kk[t^\pm])$ of the complement of $t=0$ in the total space of the $\G_{m,t}$-equivariant family
\[ \pi: \cX = \Spec(\mathrm{Rees}_\cF(R)) \longrightarrow \A^1_t =\Spec(\kk[t])\,.\]
To describe $\rho: \cY \rightarrow \A^1$, we first consider the base change 
of $\pi: \cX \rightarrow \A^1_t$ along the map $\A^2_{x,y} \rightarrow \A^1_t$ given by $t=xy$. 
By Lemma \ref{lem:isomprhismAtilde}, this base change is the 
 $\mathbb{G}_{m,x} \times \mathbb{G}_{m,y} = \Spec(\kk[x^{\pm}, y^{\pm}])$-equivariant morphism 
\[ \Spec(\overline{\mathrm{Rees}}_\cF(R)) \longrightarrow \mathbb{A}^2_{x,y} = \Spec(\kk[x,y])  \,. \]
Taking the quotient of the complement of $y = 0$ in $\Spec(\overline{\mathrm{Rees}}_\cF(R))$ 
by $\mathbb{G}_{m,y}$, we obtain the canonical degeneration of $(Y,D)$ 
given by the $\mathbb{G}_{m,x}$-equivariant morphism
\[ \rho: \cY = \Proj_s (\overline{\mathrm{Rees}}_\cF(R)) \longrightarrow \mathbb{A}^1_x = \Spec(\kk [x]) \,.  \]

\begin{remark}
The construction of degenerations via the base change $t=xy$ 
has appeared previously in several places in the literature.
For instance, it was used by the first author to construct toric degenerations of the Tate curve in \cite[Theorem 3.2]{Arguzcorals}. In a non-affine context, it was also used in non-abelian Hodge theory to construct a fiberwise compactification of the Hodge degeneration from the de Rham moduli space to the Dolbeault moduli space \cite{hausel_compactification, simpson_hodge} -- see also \cite[Lemma 6.1]{HT2003}, \cite[Theorem 7.2.1]{HLRV2011}, and \cite[Theorem 3.2]{FM2022}. 
By contrast, in this paper, we apply this construction to degenerations of character varieties, that is, of Betti moduli spaces in the terminology of non-abelian Hodge theory.
\end{remark}

\subsubsection{Geometric description of the canonical degeneration: normal cone} \label{sec_normal_cone}

Finally, we explain how the canonical degeneration
$\rho \colon (\mathcal Y,\mathcal D)\longrightarrow \A^1$
of \((Y,D)\) can be interpreted in terms of the degeneration to the weighted normal cone \cite[\S 5.1]{fultonbook}. Since this interpretation will not be used elsewhere in the paper, we restrict ourselves to a brief discussion.

Let \(\mathcal F=(F_d)_{d\ge0}\) be a projective filtration of a \(\kk\)-algebra \(R\), let \((Y,D)\) be the associated projective compactification of \(X=\Spec(R)\), and assume that \(X\) is normal and that the degree-like function \(\delta\) defining \(\mathcal F\) is a subdegree. By Proposition~\ref{prop_proj_comp}, if \(D_1,\ldots,D_\ell\) denote the irreducible components of \(D\), then the \(\QQ\)-divisor
\[
E=\sum_{i=1}^{\ell}\frac{1}{d_i}D_i,
\]
where the integers \(d_i\) are given by \eqref{eq:di}, is an ample \(\QQ\)-Cartier divisor satisfying
\[
F_d=H^0\bigl(Y,\mathcal O_Y(dE)\bigr)
\]
for every \(d\ge0\).

The degeneration of \(Y\) to the weighted normal cone of \(E\) is given by
\[
\widetilde{\rho}\colon
\widetilde{\cY}
\longrightarrow
\A^1,
\]
where $\widetilde{\cY}:=\mathrm{Bl}_{E\times\{0\}}(Y\times\A^1)$ denotes the weighted blow-up of \(E\times\{0\}\subset Y\times\A^1\), and the map $\widetilde{\rho}$ is the composition of the blow-up map with the projection $Y \times \A^1 \rightarrow \A^1$. 
Weighted blow-ups may be realized as ordinary blow-ups of suitable Deligne--Mumford stacks, obtained by introducing a root stack of order \(d_i\) along each divisor \(D_i\); see \cite{quek2021weighted} and \cite[\S~3]{abramovich2024functorial} for details. The central fiber of \(\widetilde{\rho}\) consists of two irreducible components: the weighted blow-up \(\mathrm{Bl}_E(Y)\) of \(Y\) along \(E\), and the projectivization of the weighted normal bundle of \(E\) in \(Y\).

Then, the canonical degeneration
\[
\rho\colon \mathcal Y\longrightarrow\A^1
\]
is obtained from \(\widetilde{\rho}\) by contracting the component
\(\mathrm{Bl}_E(Y)\subset\widetilde{\rho}^{-1}(0)\) to a point -- see Figure \ref{figure2}. 
This contraction is induced by the relatively nef \(\QQ\)-Cartier divisor \(E\times\A^1\) on $\widetilde{\cY}$, which descends to a relatively ample \(\QQ\)-Cartier divisor on the resulting family
\(\mathcal Y\to\A^1\).

\begin{figure}[htbp]
\center{\scalebox{.6}{\input{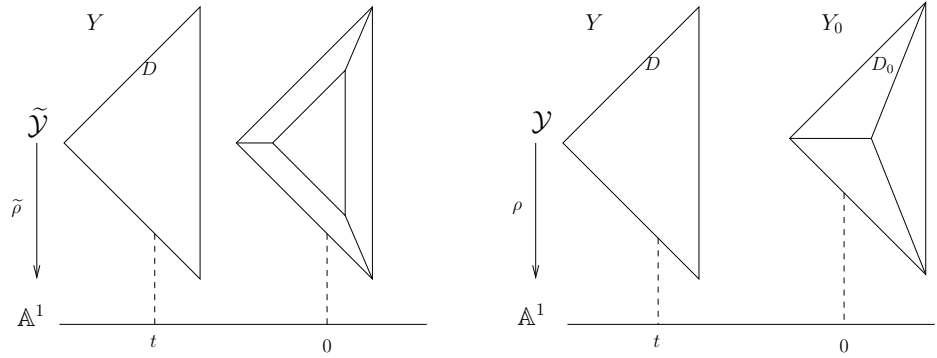}}}
\caption{The degeneration to the weighted normal cone $\widetilde{\rho}$ on the left, and the canonical degeneration $\rho$ on the right.}
\label{figure2}
\end{figure}

\begin{example} \label{ex_cusp}
Let $Y$ be a smooth projective curve of genus one, $D=p$ the divisor consisting of a single point $p\in Y$, and $X:=Y\setminus D$ the complement affine curve. The order of poles along $p$ defines a projective filtration $\mathcal F$ on the algebra of regular functions 
$R:=H^0(X,\mathcal O_X)$,
whose associated projective compactification is the pair $(Y,D)$. By the Riemann--Roch theorem and Serre duality,
we have
\[ H^0(Y,\mathcal O_Y)=H^0(Y,\mathcal O_Y(p))=\kk,\text{     and    } 
H^0(Y,\mathcal O_Y((k+1)p))/H^0(Y,\mathcal O_Y(kp))\cong \kk\text{   for every   }k\ge1\,.\] 
Consequently, the associated graded algebra is
\[
\gr_{\mathcal F}(R)
=\kk\oplus\bigoplus_{k\ge2}\kk\,t^k
\subset \kk[t]\,.
\]
Hence, $X_0=\Spec(\gr_{\mathcal F}(R))$
is the affine cusp
\[ \Spec\bigl(\kk[x,y]/(y^2-x^3)\bigr),
\]
and the corresponding canonical degeneration
$\rho:(\mathcal Y,\mathcal D)\longrightarrow\A^1$
has central fiber $
Y_0=\Proj\!\bigl(\kk[x,y,s]/(y^2s-x^3)\bigr)$, 
the cuspidal cubic in $\PP^2$.
If
$y^2s=x^3+axs^2+bs^3$
is a Weierstrass equation for $Y$ as a cubic in $\PP^2$, then the total space of the canonical degeneration is
\[
\mathcal Y=
\Proj\!\left(
\kk[x,y,s,\lambda]/
(y^2s-x^3-a\lambda^2 xs^2-b\lambda^3s^3)
\right),
\]
where $\lambda$ is the coordinate on $\A^1$.

This degeneration also admits a simple description in terms of the degeneration to the normal cone. Blowing up $D\times\{0\}\subset Y\times\A^1$ yields a family $\widetilde{\rho}: \widetilde{\cY} \rightarrow \A^1$ whose central fiber is the union of two irreducible components: a copy of $Y$ and a copy of $\PP^1$, meeting transversely at a single point. The canonical degeneration is obtained by contracting the component isomorphic to $Y$ to a point. The remaining component $\PP^1$ maps isomorphically to $Y_0$ away from the singular point, giving the normalization map
$\PP^1\longrightarrow Y_0$
of the cuspidal cubic -- see Figure \ref{fig:fig3}.

This example does not satisfy the hypotheses of Theorem~\ref{thm_main_C}. Indeed, the graded ring $\gr_{\mathcal F}(R)$ is not the algebra of a cone complex; rather, it is the monoid algebra of a non-saturated monoid. This is consistent with the fact that $X$ does not admit a log Calabi--Yau compactification. Indeed, among smooth connected affine curves, only $\G_m$ admits a log Calabi--Yau compactification.

\begin{figure}[htbp]
\center{\scalebox{.6}{\input{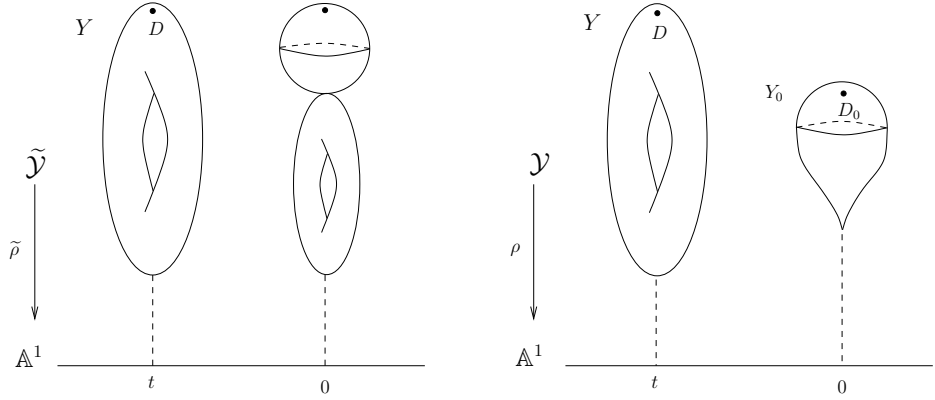}}}
\caption{The degeneration to the normal cone $\widetilde{\rho}$ and the canonical degeneration $\rho$ in Example \ref{ex_cusp} .}
\label{fig:fig3}
\end{figure}

\end{example}

\subsection{Vertices and compactified vertices}
\label{section_vertices}
In this section, we study properties of the central fibers of the canonical degenerations which will be used in the proof of Theorem \ref{thm_main_C} in \S \ref{section_end_proof}.
Recall from \S\ref{sec_algebras_cone} that every cone complex $\Omega$ determines an algebra $\kk[\Omega]$. We adopt the following terminology from \cite{GHK1}.

\begin{definition}
    The \emph{vertex} associated to the cone complex $\Omega$ is the affine scheme \[ X_0:=\Spec(\kk[\Omega])\,.\] 
The vertex $X_0$ is the union of affine toric varieties $\operatorname{Spec}(\kk[\omega])$, glued along their toric strata, where $\kk[\omega]$ is the monoid algebra of the integral points of the cones $\omega$ of $\Omega$.
\end{definition}

As reviewed in \S\ref{sec_algebras_cone}, a non-negative integral piecewise linear function $h: |\Omega| \rightarrow \RR_{\geq 0}$ defines a $\ZZ_{\geq 0}$-grading on $\kk[\Omega]$. 

\begin{definition} \label{def_comp_vertex}
    The \emph{compactified vertex} associated to the cone complex $\Omega$ and the non-negative integral piecewise linear function $h: |\Omega| \rightarrow \RR_{\geq 0}$ is the 
    projective compactification $(Y_0, D_0)$ of the vertex $X_0 = \Spec(\kk[\Omega])$ associated as in Example \ref{Example:projcone} to the $\ZZ_{\geq 0}$-grading defined by $h$ on $\kk[\Omega]$, that is, 
\[ Y_0 = \Proj \left( \bigoplus_{d \geq 0}  \bigoplus_{\substack{m \in \Omega(\ZZ) \\ 
0 \leq h(m) \leq d}} \kk z^m t^d   \right) \,,\]
    and 
    \[ D_0 = \Proj(\kk[\Omega]) \,.\]
\end{definition}

The following result shows that compactified vertices naturally occur as central fibers of canonical degenerations.

\begin{lemma}  \label{lem_central_vertex}
Let $\cF=(F_d)_{d\geq 0}$ be a projective filtration on a $\kk$-algebra $R$, $(Y,D)$
the associated projective compactification of $X=\Spec(R)$, and $\rho:(\cY, \cD) \rightarrow \A^1$ 
the canonical degeneration of $(Y,D)$. Assume that there exist a cone complex $\Omega$ and a 
non-negative integral piecewise linear function $h: |\Omega| \rightarrow \RR_{\geq 0}$ such that the corresponding algebra
$\kk[\Omega]$, endowed with the grading induced by $h$, is isomorphic to $\mathrm{gr}_\cF(R)$ as a $\ZZ_{\geq 0}$-graded algebra. 
Then, the central fiber $(Y_0,D_0)=\rho^{-1}(0)$ is the compactified vertex associated to 
$\Omega$ and $h$.
\end{lemma}

\begin{proof}  
By Proposition \ref{prop_degeneration}(iv), the central fiber $(Y_0, D_0)=\rho^{-1}(0)$ is given by 
\[  Y_0 = \Proj \left( \bigoplus_{d \geq 0} \bigoplus_{0 \leq r \leq d} F_r/F_{r-1} \right) \text{    and    }D_0=\Proj(\gr_\cF(R))=\Proj \left( \bigoplus_{d \geq 0} F_d/F_{d-1} \right) \,.\]
The assumption that $\gr_\cF(R)$ is isomorphic to $\kk[\Omega]$ with the grading defined by $h$ induces identifications 
\[ F_r/F_{r-1} \simeq \bigoplus_{\substack{m \in \Omega(\ZZ)\\ h(m)=r}} \kk z^m \,,\]
and so $(Y_0, D_0)$ is indeed the compactified vertex associated to $\Omega$ and $h$ in Definition 
\ref{def_comp_vertex}. 
\end{proof}

In the remainder of this section, we establish properties of compactified vertices. 
To do this, we use that compactified vertices are examples of stable toric varieties associated to polyhedral complexes as in \cite{Valeryannals}, as explained below.

\begin{definition}
    A \emph{compact polyhedral complex} $\Delta$ consists of a topological space $|\Delta|$, together with a finite collection of compact subspaces of $|\Delta|$, called the \emph{cells} of $\Delta$, and, for each cell $\delta$, a finitely generated group $\mathrm{Aff}_\delta$ of continuous real-valued functions on $\delta$, called the \emph{integral affine functions} on $\delta$.
    We require that $\mathrm{Aff}_\delta$ contains the group $\ZZ$ of integral constant functions. 
    Let $c: \Hom(\mathrm{Aff}_\delta, \ZZ) \rightarrow \Hom(\ZZ,\ZZ) \simeq \ZZ$ be the dual map
    to the inclusion $\ZZ \hookrightarrow \mathrm{Aff}_\delta$. Then, the following hold:
    \begin{itemize}
        \item[(i)] The evaluation map
    \[
    \phi_\delta:\delta \longrightarrow \A_\delta \otimes \RR,
    \qquad
    x\longmapsto \bigl(f\mapsto f(x)\bigr),
    \]
    is a homeomorphism from $\delta$ onto a compact lattice polytope in the real affine space $\A_\delta \otimes \RR$ endowed with the affine lattice 
    \[ \A_\delta:=\{ f \in \Hom(\mathrm{Aff}_\delta, \ZZ)\,|\, c(f)=1\}\,. \]
    \item[(ii)] For every face $\tau$ of $\phi_\delta(\delta)$, the preimage
$\rho:=\phi_\delta^{-1}(\tau)$
    is a cell of $\Delta$, and
\[ \mathrm{Aff}_\rho=\{\,f|_\rho \mid f\in \mathrm{Aff}_\delta\,\}\,.\]
    \item[(iii)] The underlying topological space $|\Delta|$ is the disjoint union of the relative interiors of the cells of $\Delta$.
    \end{itemize}
\end{definition}

\begin{definition} \label{def_cone_delta}
Let $\Delta$ be a compact polyhedral complex. The \emph{cone over $\Delta$} is the cone complex
$\mathrm{Cone}(\Delta)$
obtained by replacing each cell $\delta$ of $\Delta$ with the cone over the lattice polytope
$\delta\times\{1\}\subset \A_\delta\times\RR$. The projections $\A_\delta \times \RR \rightarrow \RR$ define a non-negative piecewise linear function $h_\Delta: |\mathrm{Cone}(\Delta)| \rightarrow \RR_{\geq 0}$ such that $h_{\Delta}^{-1}(1) \simeq |\Delta|$.
\end{definition}

\begin{definition} \label{def_stable_toric}
The \emph{stable toric variety} associated to the compact polyhedral complex $\Delta$ is the projective scheme 
\[ P[\Delta] := \Proj\left( \kk[\mathrm{Cone}(\Delta)]\right) \,,\]
where $\kk[\mathrm{Cone}(\Delta)]$ is the algebra of the cone complex $\mathrm{Cone}(\Delta)$, with $\ZZ_{\geq 0}$-grading defined by the non-negative piecewise linear function $h_\Delta$. 
The scheme $P[\Delta]$ is the union of projective toric varieties with momentum polytopes 
given by the cells of $\Delta$, glued along their toric strata. 
\end{definition}

\begin{lemma} \label{lem_1}
The compactified vertex $Y_0$ associated with the cone complex $\Omega$ and the non-negative integral
piecewise linear function $h: |\Omega| \rightarrow \RR_{\geq 0}$ coincides with the stable toric variety $P[\Delta(\Omega,h)]$ associated with the polyhedral complex $\Delta(\Omega,h)$, whose cells are
\[
\delta_\sigma=\{\,x\in\sigma\mid h(x)\le 1\,\},
\]
where $\sigma$ ranges over the cones of $\Omega$:
\[ Y_0 = P[\Delta(\Omega, h)] \,.\]
\end{lemma}

\begin{proof}
By Definition \ref{def_stable_toric}, we have $P[\Delta(\Omega,h)] = \Proj(\kk[\mathrm{Cone}(\Delta(\Omega,h))])$. By Definition \ref{def_algebra_cone}, we have 
\[\kk[\mathrm{Cone}(\Delta(\Omega,h))] = \bigoplus_{p \in \mathrm{Cone}(\Delta(\Omega,h))(\ZZ)} \kk z^p \,. \]
Using the description of $\mathrm{Cone}(\Delta(\Omega,h))$ in Definition 
\ref{def_cone_delta}, we obtain
\[\kk[\mathrm{Cone}(\Delta(\Omega,h))] = \bigoplus_{d \geq 0} \bigoplus_{\substack{
p \in h_{\Delta(\Omega,h)}^{-1}(d) \\ \cap  \mathrm{Cone}(\Delta(\Omega,h))(\ZZ)}} \kk z^p \,.\]
By the definition of $\Delta(\Omega,h)$, we have 
\[h_{\Delta(\Omega,h)}^{-1}(d) \cap  \mathrm{Cone}(\Delta(\Omega,h))(\ZZ)= \{ m \in \Omega(\ZZ)\,|\, h(m) \leq d \} \times \{d\} \,. \]
Hence, 
\[\kk[\mathrm{Cone}(\Delta(\Omega,h))] =  \bigoplus_{d \geq 0}  \bigoplus_{\substack{m \in \Omega(\ZZ) \\ 
0 \leq h(m) \leq d}} \kk z^m t^d \,,\]
and so the identification 
$Y_0 \simeq P[\Delta(\Omega, h)]$
follows from the description of the compactified vertex $Y_0$
in Definition \ref{def_comp_vertex}.
\end{proof}

The following result generalizes \cite[Lemma 17.1]{keel2024log}, which is formulated for simplicial complexes, to the more general context of polyhedral complexes. We refer to \cite{kollar2013singularities, kollar2023families-of-varieties} for the definition of semi log canonical pairs, which extend the notion of log canonical pairs to varieties that are not necessarily normal.

\begin{lemma} \label{lem_2}
Let $P[\Delta]$ be the stable toric variety associated to a compact polyhedral complex $\Delta$.
Assume that:
\begin{itemize}
    \item[(i)] The intersection of two cells of $\Delta$ is a single cell of $\Delta$.
    \item[(ii)] All the maximal cells of $\Delta$ have the same dimension $n$.
    \item[(iii)] There exists an orientation of each maximal cell such that every $(n-1)$-dimensional cell of $\Delta$ is either the intersection of exactly two maximal cells with opposite induced orientations or the face of a single maximal cell.
    \item[(iv)] $\Delta$ is locally Cohen--Macaulay of pure dimension $n$, that is, $\widetilde{H}_i(|\Delta|, |\Delta| \setminus \{u\}, \QQ)=0$ for every point $u \in |\Delta|$ and every $0 \leq i <n$.
    \item[(v)] $H^1(|\Delta|, \QQ)=0$. 
\end{itemize}
Then the following hold:
\begin{itemize}
 \item[(i)] The stable toric variety $P[\Delta]$ is reduced, Cohen--Macaulay of pure dimension $n$, normal crossing in codimension one, and satisfies $H^1(P[\Delta], \cO_{P[\Delta]})=0$. Moreover, if $\partial \Delta$ denotes the compact polyhedral complex consisting of the $(n-1)$-dimensional cells that are faces of exactly one maximal cell, the corresponding stable toric variety $P[\partial \Delta]$ is a reduced subscheme of pure codimension one in $P[\Delta]$.
    \item[(ii)] The pair $(P[\Delta], P[\partial \Delta])$ is semi log canonical.
    \item[(iii)] The sheaf $\omega_{P[\Delta]}(P[\partial \Delta])$ is isomorphic to $\cO_{P[\Delta]}$.
\end{itemize}
\end{lemma}

\begin{proof}
We follow the structure of the proof of \cite[Lemma 17.1]{keel2024log}. Since every toric variety is reduced, so is the stable toric variety \(P[\Delta]\). Assumptions~(ii) and~(iii) imply that \(P[\Delta]\) is equidimensional of pure dimension \(n\), has normal crossing singularities in codimension one, and that \(P[\partial\Delta]\) is a reduced subscheme of pure codimension one. By assumption~(i), \(\Delta\) is a semilattice in the sense of \cite{Valeryannals}.
Hence \cite[Lemma 2.4.12]{Valeryannals} shows that assumption~(iv), that \(\Delta\) is locally Cohen--Macaulay, implies that \(P[\Delta]\) is Cohen--Macaulay. Finally, we have $H^i(P[\Delta], \cO_{P[\Delta]})=H^i(|\Delta|, \QQ) \otimes \kk$ by \cite[Theorem 2.5.6]{Valeryannals}, and so assumption (v) implies $H^1(P[\Delta], \cO_{P[\Delta]})=0$. This proves~(i).

The orientation of \(\Delta\) from assumption~(iii) determines compatible log volume forms on the toric irreducible components, which glue along the normal crossing codimension one strata to give a trivialization of
$\omega_{P[\Delta]}(P[\partial\Delta])$ in codimension one. Since \(P[\Delta]\) is Cohen--Macaulay, it is in particular $S_2$, and the reflexive sheaf \(\omega_{P[\Delta]}(P[\partial\Delta])\) is therefore \(S_2\). Hence the trivialization extends uniquely across the codimension-two locus, yielding
$\omega_{P[\Delta]}(P[\partial\Delta])\simeq \mathcal{O}_{P[\Delta]}$. This proves (iii).

Finally, \(P[\Delta]\) is reduced, equidimensional, \(S_2\), and normal crossing in codimension one, while \(\omega_{P[\Delta]}(P[\partial\Delta])\) is invertible by~(iii). Moreover, the normalization of \(P[\Delta]\) is the disjoint union of its toric irreducible components, each endowed with its toric boundary, and every such pair is log canonical by \cite[Corollary 11.4.25]{CLS}. Therefore, the pair \((P[\Delta],P[\partial\Delta])\) is semi log canonical, proving~(ii).
\end{proof}

\begin{remark}
In the proof of Lemma~\ref{lem_2}, we use \cite[Lemma~2.4.12]{Valeryannals} to conclude that $P[\Delta]$ is Cohen--Macaulay. The proof of \cite[Lemma~2.4.12]{Valeryannals} (see also \cite[Theorem~2.3.19]{Valeryannals}) generalizes the argument in the proof of \cite[\S4.6]{stanley1987generalized}, which in turn is based on the work of \cite{yuzvinsky1987cohen}. The latter develops a general framework for establishing Cohen--Macaulay properties of ``glued schemes.'' In the special case where $\Delta$ is a simplicial complex, corresponding to the classical setting of Stanley--Reisner rings, this characterization was first proved in \cite[Theorem~1]{reisner1976cohen}.

The definition of local Cohen--Macaulay for a polyhedral complex $\Delta$ in Lemma~\ref{lem_2}, 
given by the cohomological vanishing
$\widetilde{H}_i\bigl(|\Delta|,|\Delta|\setminus\{u\},\QQ\bigr)=0$
for every point $u\in |\Delta|$ and every $0\le i<n$, is a combinatorial analogue of Grothendieck's cohomological characterization of Cohen--Macaulay schemes: 
if $Y$ is a scheme of pure dimension $n$, then $Y$ is Cohen--Macaulay at a point $y\in Y$ if and only if
$H^i_y(Y,\mathcal{O}_Y)=0$
for all $0\le i<n$, where $H^i_y$ denotes local cohomology \cite[Theorem~3.8]{Hartshorne1967local}.
\end{remark}

In the following result, we consider a cone complex that is homeomorphic to a cone over a sphere. The associated compact polyhedral complex is therefore homeomorphic to a ball, and hence the cohomological vanishing assumptions of Lemma \ref{lem_2} are automatically satisfied.

\begin{lemma} \label{lem_comp_vertex}
 Let $(Y_0,D_0)$ be the compactified vertex associated to a cone complex $\Omega$ and a non-negative integral piecewise linear function $h: |\Omega| \rightarrow \RR_{\geq 0}$. Assume that the intersection of any two cones of $\Omega$ is a single cone of $\Omega$, and that the level set $h^{-1}(1)$ is homeomorphic to a sphere of dimension $n-1$. 
Then, the following hold:
\begin{itemize}
    \item[(i)] The scheme $Y_0$ is reduced, Cohen--Macaulay of pure dimension $n$, normal crossing in codimension one, and satisfies $H^1(Y_0, \cO_{Y_0})=0$. Moreover, $D_0 \subset Y_0$ is a reduced subscheme of pure codimension one. 
    \item[(ii)] The pair $(Y_0, D_0)$ is semi log canonical.
    \item[(iii)] The sheaf $\omega_{Y_0}(D_0)$ is isomorphic to $\cO_{Y_0}$.
\end{itemize}
\end{lemma}

\begin{proof}
By Lemma~\ref{lem_1}, we have \(Y_0=P[\Delta(\Omega,h)]\). Since the intersection of any two cones of \(\Omega\) is a cone of \(\Omega\), it follows that the intersection of any two cells of \(\Delta(\Omega,h)\) is a cell of \(\Delta(\Omega,h)\). Thus, condition~(i) of Lemma~\ref{lem_2} is satisfied.

Furthermore, \(h^{-1}(1)\) is homeomorphic to the \((n-1)\)-sphere, so the pair \((|\Delta(\Omega,h)|,|\partial\Delta(\Omega,h)|)\) is homeomorphic to the pair consisting of the \(n\)-ball and its boundary. In particular, conditions~(ii)--(iv) of Lemma~\ref{lem_2} hold for \(\Delta(\Omega,h)\) 
since the ball is orientable and the link of every point of $|\Delta(\Omega, h)|$ is a sphere for an interior point and a half-sphere for a boundary point.
Moreover, condition (v) of Lemma \ref{lem_2} also holds since $|\Delta(\Omega,h)|$ is a ball. 
Applying Lemma~\ref{lem_2} to $\Delta(\Omega,h)$ therefore yields the desired conclusions.
\end{proof}

\subsection{Extending properties from the central fiber to general fibers}
\label{section_extending}

In this section, we show how properties of a general fiber in a degeneration can be deduced from the corresponding properties of the central fiber.

\begin{definition}
A \emph{family of pairs} $\rho: (\cY , \cD) \rightarrow \A^1$ is a flat projective morphism of schemes $\rho: \cY \rightarrow \A^1$, together with a closed subscheme $\cD \subset \cY$ of pure codimension one such that $\rho|_\cD: \cD \rightarrow \A^1$ is flat.
\end{definition}

The following result generalizes \cite[Lemma~8.33]{GHKK} (Lemma~8.42 in the arXiv version), originally proved for normal varieties, to the broader setting of semi log canonical varieties. Similar statements have also been established in \cite[Proposition~7.1]{HKY20}, \cite[Theorem~6.36]{AAB2024ksba}, and \cite[Theorem~17.5]{keel2024log}.

\begin{lemma} \label{lem_central_general_technical}
Let
$\rho: (\cY, \cD) \rightarrow \A^1$ be a family of pairs, with central fiber $(Y_0, D_0):=\rho^{-1}(0)$. 
Assume that:
\begin{itemize}
    \item[(i)] The scheme $Y_0$ is Cohen--Macaulay and $H^1(Y_0, \cO_{Y_0})=0$.
    \item[(ii)] The pair $(Y_0, D_0)$ is semi log canonical.
    \item[(iii)] The sheaf $\omega_{Y_0}(D_0)$ is isomorphic to $\cO_{Y_0}$.
\end{itemize}
Then, there exists a Zariski open subset $U \subset \A^1$ containing $0$, such that, for every $t\in U$, denoting  
$(Y_t, D_t):=\rho^{-1}(t)$,
$Y_t$ is Cohen--Macaulay, the pair $(Y_t, D_t)$ is semi log canonical and $\omega_{Y_t}(D_t) \simeq \cO_{Y_t}$.
\end{lemma}

\begin{proof}
We follow the structure of the proof of \cite[Lemma~8.33]{GHKK} (Lemma~8.42 in the arXiv version).
Since $Y_0$ is Cohen--Macaulay by (i) and $\omega_{Y_0}(D_0)\simeq\mathcal O_{Y_0}$ by (iii), \cite[Corollary~2.71]{kollar2013singularities} implies that
\[
\mathcal O_{Y_0}(-D_0)
\simeq
\mathcal Hom_{\mathcal O_{Y_0}}\!\bigl(\omega_{Y_0}(D_0),\omega_{Y_0}\bigr)
\]
is Cohen--Macaulay. It then follows from \cite[Corollary~2.63]{kollar2013singularities} that $D_0$ is Cohen--Macaulay.

Since $\rho:\mathcal Y\to\A^1$ is flat and projective, \cite[Theorem~12.2.4]{grothendieck1966EGAIV} shows that, after replacing $\A^1$ by a Zariski open neighborhood of $0$, the fibers $Y_t$ and $D_t$ are Cohen--Macaulay for every $t$. Applying again \cite[Corollaries~2.63 and~2.71]{kollar2013singularities}, we deduce that
$\mathcal O_{Y_t}(-D_t)$
and consequently $
\omega_{Y_t}(D_t)
\simeq
\mathcal Hom_{\mathcal O_{Y_t}}
\bigl(\mathcal O_{Y_t}(-D_t),\omega_{Y_t}\bigr)$
are Cohen--Macaulay.

Since $\A^1$ is smooth and hence Cohen--Macaulay, \cite[\href{https://stacks.math.columbia.edu/tag/045J}{Tag~045J}]{stacks-project} implies, after further shrinking the base if necessary, that the total spaces $\mathcal Y$ and $\mathcal D$ are Cohen--Macaulay. Another application of \cite[Corollaries~2.63 and~2.71]{kollar2013singularities} therefore shows that $\omega_{\mathcal Y}(\mathcal D)$ is Cohen--Macaulay. Since
\[
\omega_{\mathcal Y/\A^1}
=
\omega_{\mathcal Y}\otimes(\rho^*\omega_{\A^1})^\vee
\simeq
\omega_{\mathcal Y},
\]
we conclude that $\omega_{\mathcal Y/\A^1}(\mathcal D)$ is Cohen--Macaulay as well.

Applying \cite[\href{https://stacks.math.columbia.edu/tag/045J}{Tag~045J}]{stacks-project}, we obtain that
$\omega_{\mathcal Y/\A^1}(\mathcal D)\big|_{Y_0}$ is Cohen--Macaulay.
Since $(Y_0,D_0)$ is semi log canonical by (ii), $Y_0$ is normal crossing in codimension one, and so the natural restriction morphism
\[
\omega_{\mathcal Y/\A^1}(\mathcal D)\big|_{Y_0}
\longrightarrow
\omega_{Y_0}(D_0)
\simeq
\mathcal O_{Y_0}.
\]
is an isomorphism in codimension one. As both source and target are Cohen--Macaulay, and so $S_2$ sheaves, it follows that this morphism is an isomorphism $
\omega_{\mathcal Y/\A^1}(\mathcal D)\big|_{Y_0}
\simeq
\mathcal O_{Y_0}$. As $\omega_{\mathcal Y/\A^1}(\mathcal D)$ is flat over $\A^1$, it follows that  $\omega_{\mathcal Y/\A^1}(\mathcal D)$ is invertible over a Zariski open neighborhood of $0$.
Since $H^1(Y_0,\mathcal O_{Y_0})=0$ and $\rho$ is flat and projective, we also have $H^1(Y_t, \cO_{Y_t})=0$ for $t$ in a Zariski open neighborhood of $0$ by upper semicontinuity of coherent cohomology \cite[Theorem III.12.8]{Hartshorne}. 
Hence, the connected component of the identity in the relative Picard scheme is an isomorphism over $\A^1$, and so we have
$\omega_{\mathcal Y/\A^1}(\mathcal D)
\simeq
\mathcal O_{\mathcal Y}$. 

For $t$ in a Zariski open neighborhood of $0$, the fibers $Y_t$ remain normal crossing in codimension one by the local deformation theory of the node $xy=0$. Hence the natural restriction morphism
\[
\mathcal O_{Y_t}
\simeq
\omega_{\mathcal Y/\A^1}(\mathcal D)\big|_{Y_t}
\longrightarrow
\omega_{Y_t}(D_t)
\]
is an isomorphism in codimension one. Since both $\mathcal O_{Y_t}$ and $\omega_{Y_t}(D_t)$ are Cohen--Macaulay, and therefore $S_2$, the morphism is an isomorphism:
$
\omega_{Y_t}(D_t)
\simeq
\mathcal O_{Y_t}$. 

Finally, $(Y_0,D_0)$ is semi log canonical, the divisor $\mathcal D$ is flat over $\A^1$, and the sheaf $\omega_{\mathcal Y/\A^1}(\mathcal D)\simeq\mathcal O_{\mathcal Y}$ is invertible.
Therefore, \cite[Corollary~4.45]{kollar2023families-of-varieties} implies that $(Y_t,D_t)$ is semi log canonical for all $t$ in a Zariski open neighborhood of $0$.
\end{proof}

\begin{lemma} \label{lem_central_general}
Let
$\rho: (\cY, \cD) \rightarrow \A^1$ be a $\G_m$-equivariant family of pairs, with central fiber $(Y_0, D_0):=\rho^{-1}(0)$, and general fiber $(Y,D):=\rho^{-1}(t)$ for any  $t \in \A^1 \setminus \{0\}$. Assume that:
\begin{itemize}
    \item[(i)] The scheme $Y_0$ is Cohen--Macaulay and $H^1(Y_0, \cO_{Y_0})=0$.
    \item[(ii)] The pair $(Y_0, D_0)$ is semi log canonical.
    \item[(iii)] The sheaf $\omega_{Y_0}(D_0)$ is isomorphic to $\cO_{Y_0}$.
    \item[(iv)] The scheme $Y$ is normal.
\end{itemize}
Then, $Y$ is Cohen--Macaulay, and the pair $(Y,D)$ is log canonical and log Calabi--Yau, that is, $K_Y +D \sim 0$.
\end{lemma}

\begin{proof}
By Lemma \ref{lem_central_general_technical}, \(Y\) is Cohen--Macaulay, the pair \((Y,D)\) is semi log canonical, and \(\omega_Y(D) \simeq \mathcal{O}_Y\). As \(Y\) is normal by assumption, the semi log canonical condition is equivalent to log canonicity, and hence \((Y,D)\) is log canonical.
\end{proof}

\subsection{End of the proof of the  log Calabi--Yau criterion}
\label{section_end_proof}

In this final section, we combine the results of 
\S \ref{section_degeneration}-\ref{section_extending}
to  prove the log Calabi--Yau criterion given by Theorem \ref{thm_main_C}.

Let $\cF$ be a projective filtration on a $\kk$-algebra $R$, satisfying the assumptions of Theorem \ref{thm_main_C}, and let $(Y,D)$ be the associated projective compactification of $X=\Spec(R)$. Conclusion (i) of Theorem \ref{thm_main_C} follows from assumptions (i) and (ii) by Proposition \ref{prop_proj_comp}. 
In order to prove conclusions (ii) and (iii), we consider the canonical degeneration $\rho: (\cY, \cD) \rightarrow \A^1$ of $(Y,D)$ given by Definition \ref{def_canonical_degeneration}. 

By assumption (iii), there exist a cone complex $\Omega$ and a non-negative integral piecewise linear function $h: |\Omega| \rightarrow \RR_{\geq 0}$ such that the corresponding algebra $\kk[\Omega]$ endowed with the grading defined by $h$ is isomorphic to 
$\gr_\cF(R)$ as a graded algebra. Hence, the central fiber $(Y_0, D_0)$ of the canonical degeneration is the compactified vertex 
associated to $\Omega$ and $h$ by Lemma \ref{lem_central_vertex}.

Moreover, assumption (iii) also implies that the intersection of two cones of $\Omega$ is a cone of $\Omega$, and the level set $h^{-1}(1)$ is homeomorphic to a sphere. 
Therefore, Lemma \ref{lem_comp_vertex} applies to the compactified vertex $(Y_0, D_0)$, and it follows that $Y_0$ is Cohen--Macaulay and that the pair $(Y_0, D_0)$ is semi log canonical and log Calabi--Yau.
Finally, Lemma \ref{lem_central_general} shows that these properties deform from the central fiber $(Y_0,D_0)$ 
to the general fiber $(Y,D)$ of the canonical degeneration, that is, $Y$ is also Cohen--Macaulay, and the pair $(Y,D)$ is log canonical and log Calabi--Yau. This proves conclusions (ii)--(iii), and hence completes the proof of  Theorem \ref{thm_main_C}.

\bibliographystyle{plain}
\bibliography{bibliography}
%-------------------------------------------------------------------------------

\end{document}